\documentclass[preprint,12pt]{elsarticle}

\usepackage{amssymb}
\usepackage{amsmath}
\allowdisplaybreaks
\usepackage{amsthm}

\usepackage{ifthen}

\usepackage{appendix}

\usepackage{nicematrix}
\usepackage{array}
\usepackage{multirow}

\usepackage{titlesec}

\titleformat{\section}
  {\normalfont\bfseries}
  {\thesection}{1em}{}        

\titleformat{\subsection}
  {\normalfont}
  {\thesubsection}{1em}{}

\titleformat{\subsubsection}
  {\normalfont}
  {\thesubsubsection}{1em}{}

\numberwithin{equation}{section}

\newtheorem{Definition}{Definition}[section]
\newtheorem{Notation}[Definition]{Notation}
\newtheorem{Lemma}{Lemma}[section]
\newtheorem{Theorem}[Definition]{Theorem}
\newtheorem{Proposition}[Definition]{Proposition}

\newtheorem{Example}[Definition]{Example}
\newtheorem{Corollary}[Definition]{Corollary}

\newtheorem{apxlemma}{Lemma}[section]

\newcounter{arxiv}
\ifthenelse{\value{arxiv} = 0}{
\input 00-01_specific_journal_20260130a.tex
}

\begin{document}

\begin{frontmatter}



\title{Estimating the spectral radius of Bell-type operator via finite dimensional approximation of orthogonal projections}


\author[melco]{Yuki Fujii\corref{cor1}} 
\ead{Fujii.Yuki@df.MitsubishiElectric.co.jp}
\cortext[cor1]{Corresponding author.}

\author[melco]{Toyohiro Tsurumaru} 

\affiliation{organization={Mitsubishi Electric Corporation, Information Technology R\&D Center},
            addressline={5-1-1 Ofuna}, 
            city={Kamakura-shi},
            postcode={247-8501}, 
            state={Kanagawa},
            country={Japan}}

\begin{abstract}

We establish a new decomposition formula for two orthogonal projections $P$ and $Q$ on a separable Hilbert space $V$. This formula yields an orthogonal direct sum decomposition of $V$ into invariant subspaces under $P$ and $Q$, each of which is either at most two dimensional or infinite dimensional. On every infinite dimensional component, the pair $(P,Q)$ admits a matrix representation that we call the ``one-shifted form''. This representation diagonalizes both $P$ and $Q$ into blocks of size at most two, and moreover,
both projections can be explicitly approximated by orthogonal projections on finite dimensional subspaces. This approximation scheme offers a way to derive infinite dimensional results from their finite dimensional counterparts and is also useful in numerical computations.

This decomposition provides a useful framework for analyzing a wide
range of problems involving two orthogonal projections in infinite
dimensions. 
In particular, several spectral problems for operators generated by $P$ and $Q$
(including polynomials in $P$ and $Q$)
can be reduced to the case where the pair admits a one-shifted form.

More concretely, we can estimate the spectral radius $\rho([P,Q])$, which is equivalent to estimating the spectral radius of the Bell-CHSH operator, a quantity of fundamental importance in quantum mechanics.
We provide an upper bound and a lower bound for $\rho([P,Q])$, which become exact when the matrix representations of $P$ and $Q$ are in ``constant-angle one-shifted form''.
\end{abstract}


\ifthenelse{\value{arxiv} = 0}{
\input 00-02_specific_journal_20260130a.tex
}

\begin{keyword}
Orthogonal projections \sep Infinite dihedral group \sep Free product of cyclic groups of order 2 \sep Bell-CHSH operator \sep Tsirelson's inequality \sep Device-independent quantum key distribution \sep Tridiagonal matrix
\end{keyword}

\end{frontmatter}

\ifthenelse{\value{arxiv} = 1}{
\tableofcontents
}

\noindent \textbf{MSC:} 20C07, 47A25

\section{Introduction}
\label{sec:intro}
The problem of studying properties of orthogonal projections on a separable Hilbert space $V$ arises in various fields, including quantum mechanics, and has been extensively studied from a mathematical perspective \cite{ABottcher2010}.

For instance, given two orthogonal projections $P$ and $Q$ on $V$, let
$A = 2P - \operatorname{id}_V, B = 2Q - \operatorname{id}_V$, and
consider the problem of computing the spectral radius $\rho([A,B])$ of the commutator $[A,B]$, defined by\\
\begin{equation*}
\rho([A,B]) = \sup_{u \in V, \|u\| = 1} |([A,B]u, u)|.
\end{equation*}

This problem is known to be equivalent to determining the spectral radius of the Bell-CHSH operator \cite{AChefles1997},
which is a bounded self-adjoint operator on a Hilbert tensor space $V_1 \otimes V_2$ used in quantum mechanics, as shown in \cite{LAKhalfin1985, LJLandau1987}. For the definition of the Bell-CHSH operator and the background knowledge assumed in this paper, see Section~\ref{subsec:chshoperatorintroduction}.

The problem of determining the spectral radius of the Bell-CHSH operator has important applications, such as testing the violation of local realism in quantum mechanics \cite{JFClauser1969}, and evaluating the security of device-independent quantum key distribution protocols \cite{SPironio2009}. Owing to these applications, the computation of $\rho([A,B])$ is important.

When $V$ is finite dimensional, this problem can be solved immediately by the existing results. 
First, it is known that 
there is a one-to-one correspondence between pairs of orthogonal projections on $V$ and unitary representations of the infinite dihedral group $D_\infty$ on $V$.
Using group representation theory,
Behncke \cite{HBehncke1971} showed that, regardless of whether $V$ is finite or infinite dimensional, $V$ is irreducible only if $\dim V \leq 2$. Here, $V$ is said to be irreducible if every invariant subspace under $P$ and $Q$ is trivial.

By combining this result with the general theory of decomposition of group representations into irreducible components, it follows that if $V$ is finite dimensional, then $V$ can be decomposed as an orthogonal direct sum of invariant subspaces under $P$ and $Q$, each of which is at most two dimensional.
In quantum information theory, this fact is known as Jordan's lemma \cite{Primaatmaja2023, MMcKague2012}.

Using this decomposition, the problem of computing $\rho([A,B])$ in finite dimensional $V$ reduces to the case where $\dim V \leq 2$. In these cases, it is known that $P$ and $Q$ admit a simple matrix representation using Halmos' theorem \cite{PRHalmos1969}. Based on this representation and elementary calculations, we find that $\rho([A,B])$ can be computed using the eigenvalues of $A + B$.

From these observations, we obtain the following proposition.

\begin{Proposition}
\label{prop:commutorspectrum_finitedimension}
Assume that $V$ is finite dimensional and let $\lambda$ be an eigenvalue of $A + B$ whose square is closest to $2$. Then
\begin{equation}
\label{eq:commutorspectrum_finitedimension}
\rho([A,B]) = b(\lambda),
\end{equation}
where $b$ is defined by
$
\displaystyle b(\lambda) = 2\sqrt{1 - \left( \dfrac{\lambda^2}{2} - 1 \right)^2}\quad(\lambda \in \mathbb{R})
$.
\end{Proposition}

On the other hand, when $V$ is infinite dimensional, it is not known whether or how equation (\ref{eq:commutorspectrum_finitedimension}) can be extended, either in the form of an equality or an inequality. One reason for this is the existence of pairs of orthogonal projections on an infinite dimensional Hilbert space that admit no nonzero finite dimensional invariant subspaces under both projections \cite{GAllan1998}.

In this paper, we consider the problem of extending equation (\ref{eq:commutorspectrum_finitedimension}) in the form of an inequality as a representative example. As in this case, there are several known problems concerning the properties of orthogonal projections that have only been resolved in the finite dimensional case.

This work makes three main contributions.
First, we prove that every finite family of orthogonal projections on an infinite dimensional Hilbert space can be approximated, in the sense of Theorem~\ref{theorem:oneshifterform}, by orthogonal projections on finite dimensional subspaces. 
For details of this result, see Section~\ref{subsec:finitedimensionalapproximationbasicproperties}.
This approximation allows certain inequalities to be reduced to the finite dimensional case. 
This method is stated in Section~\ref{subsec:reduction}. 

As a simple application, in Section~\ref{subsec:tsirelson}, we present  a new proof of Tsirelson's inequality \cite{BSTsirelson1980}, which asserts that the spectral radius of every Bell-CHSH operator is bounded above by $2\sqrt{2}$.
We compare our proof with others.
When either $V_1$ or $V_2$ has dimension at most two, Tsirelson's inequality can be proved using elementary calculus.
Tsirelson~\cite{BSTsirelson1987} showed that this inequality can be reduced to the case of dimension at most two by an ingenious technique using the Clifford algebra. 
Based on our approximation method and Jordan's lemma, we reduce the proof to this case.

Second, for two orthogonal projections $P$ and $Q$, we provide
an explicit formula for a finite dimensional approximation of the pair $(P,Q)$.
This can be summarized as Theorem~\ref{theorem:oneshifterform}.
For details of this theorem, see Section~\ref{sec:two}.

\begin{Theorem}
\label{theorem:oneshifterform}
There exists an orthogonal direct sum decomposition of $V$ into invariant subspaces under $P$ and $Q$, 
each of which is either at most two dimensional or infinite dimensional. Every such infinite dimensional component
 admits an explicit finite dimensional approximation of the pair $(P,Q)$ in the following sense.
There exists an orthonormal basis $\mathcal{U} = \{u_i\}_{i=1}^\infty$ and a pair of sequences
$\Theta = \{\theta_i\}_{i=1}^\infty, \Omega = \{\omega_i\}_{i=1}^\infty \subset (0,\pi)$ such that, for any vector $v \in V$,
\begin{equation*}
\max\{\|v - \pi_n v\|, \|Pv - P_n(\Theta)\pi_n v\|, \|Qv - Q_n(\Omega)\pi_n v\|\} \to 0 \quad (n \to \infty),
\end{equation*}
where $\pi_n$ is the orthogonal projection onto the finite dimensional subspace $V_n$ spanned by $\{u_1, \dots, u_{2n}\}$, and $P_n(\Theta)$, $Q_n(\Omega)$ are orthogonal projections on $V_n$.

Moreover, the matrix representations of $P_n(\Theta)$ and $Q_n(\Omega)$ in this orthonormal basis are symmetric tridiagonal matrices given by:
\begin{align*}
2P_n(\Theta) - E_{2n} &=
\left(\begin{array}{>{\centering\arraybackslash}p{0.3cm}>{\centering\arraybackslash}p{0.3cm}>{\centering\arraybackslash}p{0.3cm}>{\centering\arraybackslash}p{0.3cm}>{\centering\arraybackslash}p{0.3cm}}
\cline{1-2}
\multicolumn{2}{|c|}{\multirow{2}{*}{$r(\theta_1)$}} & \multicolumn{1}{c}{} & \multicolumn{2}{c}{\multirow{2}{*}{\mbox{\Large $O$}}}\\
\multicolumn{1}{|c}{} & \multicolumn{1}{c|}{} & \multicolumn{1}{c}{} & & \\ \cline{1-2}
\multicolumn{2}{c}{} & $\ddots$ & \multicolumn{2}{c}{}\\
 \cline{4-5}
\multicolumn{2}{c}{\multirow{2}{*}{\mbox{\Large $O$}}} & \multicolumn{1}{c}{\multirow{2}{*}{}} & \multicolumn{2}{|c|}{\multirow{2}{*}{$r(\theta_n)$}}\\
& \multicolumn{1}{c}{} & \multicolumn{1}{c}{} & \multicolumn{1}{|c}{} & \multicolumn{1}{c|}{} \\ \cline{4-5}
\end{array}\right),\\
2Q_n(\Omega) - E_{2n} &=
\left(
\begin{array}{>{\centering\arraybackslash}p{0.3cm}>{\centering\arraybackslash}p{0.3cm}>{\centering\arraybackslash}p{0.3cm}>{\centering\arraybackslash}p{0.3cm}>{\centering\arraybackslash}p{0.3cm}>{\centering\arraybackslash}p{0.3cm}>{\centering\arraybackslash}p{0.3cm}}
\cline{1-1}
\multicolumn{1}{|c|}{\multirow{1}{*}{$1$}} & $0$ & $0$ & \multicolumn{1}{c}{\multirow{3}{*}{}} & \multicolumn{3}{c}{\multirow{3}{*}{\mbox{\Large $O$}}}\\ \cline{1-3}
\multicolumn{1}{c}{0}  & \multicolumn{2}{|c|}{\multirow{2}{*}{$r(\omega_1)$}} & \multicolumn{1}{c}{} & & &  \\
\multicolumn{1}{c}{0} & \multicolumn{1}{|c}{} & \multicolumn{1}{c|}{} & \multicolumn{1}{c}{} &  &  &  \\  \cline{2-3}
\multicolumn{3}{c}{} & $\ddots$ & \multicolumn{3}{c}{} \\ \cline{5-6}
\multicolumn{3}{c}{\multirow{3}{*}{\mbox{\Large $O$}}} & \multicolumn{1}{c}{\multirow{3}{*}{}} & \multicolumn{2}{|c|}{\multirow{2}{*}{$r(\omega_{n-1})$}} & \multicolumn{1}{c}{0} \\
 & & \multicolumn{1}{c}{} & \multicolumn{1}{c}{} & \multicolumn{2}{|c|}{} & \multicolumn{1}{c}{0} \\ \cline{5-7}
 & & \multicolumn{1}{c}{} & \multicolumn{1}{c}{} & \multicolumn{1}{c}{0} & \multicolumn{1}{c}{0} & \multicolumn{1}{|c|}{\multirow{1}{*}{$-1$}} \\ \cline{7-7}
\end{array}
\right),
\end{align*}
where $r$ is defined by
$r(x) = \begin{pmatrix} \cos x & \sin x \\ \sin x & -\cos x \end{pmatrix} \quad (x \in \mathbb{R})$.
\end{Theorem}

We refer to the infinite dimensional matrix representations admitting this approximation as the \emph{one-shifted form}. 
For the definition of one-shifted form, see Definition~\ref{def:oneshiftedform}.

Theorem~\ref{theorem:oneshifterform} is the main result of this paper.
From the viewpoint of group representation theory, two orthogonal projections having a one-shifted form yield a cyclic unitary representation of $D_\infty$. Theorem~\ref{theorem:oneshifterform} thus provides an explicit formula that decomposes every infinite dimensional unitary representation of $D_\infty$ into cyclic representations.
In other words, the one-shifted form (Definition~\ref{def:oneshiftedform}) extends Jordan's lemma to the infinite dimensional case. 
As a consequence, a variety of problems concerning infinite dimensional unitary representations of $D_\infty$
(two orthogonal projections in an infinite dimensional Hilbert space)—such as studying the spectrum of operators given by polynomials in $P$ and $Q$—can be reduced to the case where the pair admits a one-shifted form.

Thus we can reduce the problem of estimating $\rho([A,B])$
 to the case where $P$ and $Q$ admit a one-shifted form.

Third, 
when $P$ and $Q$ admit a one-shifted form,
we derive the inequalities as stated in Theorem~\ref{theorem:commutorspectrum_general}. This theorem estimates the spectral radius of $[A,B]$ from above and below using the eigenvalues of symmetric tridiagonal matrices approximating $A + B$, based on Theorem~\ref{theorem:oneshifterform}. Theorem~\ref{theorem:commutorspectrum_general} is an extension of equation (\ref{eq:commutorspectrum_finitedimension}) to infinite dimensional Hilbert spaces.

\begin{Theorem}
\label{theorem:commutorspectrum_general}
Assume that $P$ and $Q$ admit the approximation as in {\rm{Theorem~\ref{theorem:oneshifterform}}}. Then we have
\begin{equation}
\label{eq:theorem:commutorspectrum_general}
\sup_{\lambda \in F(A+B) \cup \{0\}}b(\lambda) \le \rho([A,B]) \le \liminf_{n \to \infty}b(\lambda_{n}),
\end{equation}
where for $n \in \mathbb{N}$, we define $A_n=2P_n-\operatorname{id}_{V_{n}}, B_n=2Q_n-\operatorname{id}_{V_{n}}$, and let $\lambda_n$ be an eigenvalue of $A_n+B_n$ such that $\lambda_n^2$ is closest to $2$.

The set $F(A+B)$ consists of those $\lambda \in (0,2)$ for which there exist a sequence $\{\eta_k\}_{k=1}^\infty$ and vectors $v_{-,k},v_{+,k} \in V_{k}$ $(k=1,2,\dots)$ satisfying the following conditions:
\begin{enumerate}
\item[(i)] For any $k \in \mathbb{N}$, $-\eta_k$ and $\eta_k$ are eigenvalues of the symmetric tridiagonal matrix $A_k+B_k$, and $v_{-,k}, v_{+,k}$ are unit eigenvectors corresponding to $-\eta_k$ and $\eta_k$, respectively.
\item[(ii)] $\displaystyle \lim_{k \to \infty}\eta_k = \lambda, \quad \lim_{k \to \infty}\max\{\|(P-P_k) v_{\pm,k}\|,\|(Q-Q_k) v_{\pm,k}\|\}=0$.
\end{enumerate}
\end{Theorem}

These inequalities are useful because there exist concrete examples in which equality holds. 
One such case, in which both $\Theta$ and $\Omega$ are equal to a constant $\theta$ in the one-shifted form, is referred to as a \emph{constant-angle one-shifted form}.

In this setting, we can compute the exact value of $\rho([A,B])$ using Theorem~\ref{theorem:commutorspectrum_general} and known results for symmetric tridiagonal matrices:
\begin{equation*}
\rho([A,B])=
\left\{
\begin{array}{ll}
2|\sin(2\theta)| & \quad \left|\theta-\dfrac{\pi}{2}\right| > \dfrac{\pi}{4}, \vspace{2mm} \\
2 & \quad \left|\theta-\dfrac{\pi}{2}\right| \le \dfrac{\pi}{4}.
\end{array}
\right.
\end{equation*}
For the results on symmetric tridiagonal matrices used in this paper, see Section~\ref{subsec:tridiagonalmatrixintroduction}.

As shown in Section~\ref{subsec:jordan}, there exists an example in which the orthogonal projections $P$ and $Q$ have such a one-shifted form but do not admit nonzero finite dimensional invariant subspaces under $P$ and $Q$. For instance, this occurs if $\theta=\pi/2$.

\ifthenelse{\value{arxiv} = 1}{
After the submission of the first version of this preprint \cite{YFujii2025}, B\"ottcher and Spitkovsky \cite{ABottcher2026} obtained an infinite dimensional extension of Proposition~\ref{prop:commutorspectrum_finitedimension}, 
from which a part of our Theorem~\ref{theorem:commutorspectrum_general} can be alternatively derived.
}{
After the submission of the preprint \cite{YFujii2025} of this paper and after the submission of this paper itself, B\"ottcher and Spitkovsky \cite{ABottcher2026} obtained an infinite dimensional extension of Proposition~\ref{prop:commutorspectrum_finitedimension}, 
from which a part of our Theorem~\ref{theorem:commutorspectrum_general} can be alternatively derived.
}
Furthermore, also in \cite{ABottcher2026}, they kindly pointed out minor errors in the earlier version of our preprint and helped us correct them, for which we are sincerely grateful.

\section{Preliminaries}
\label{sec:preliminary}

\subsection{Notations}
\label{subsec:notation}

Throughout this paper, we adopt the following notation.
\begin{itemize}
\item $\mathbb{Z}_2: \mathbb{Z}/2\mathbb{Z}$.
\item $\Gamma_m:$ The free product
of $m$ copies of $\mathbb{Z}_2$, often denoted as $\underbrace{\mathbb{Z}_2 * \cdots * \mathbb{Z}_2}_{\text{$m$ times}}$ for $m \in \mathbb{N}$. Elements $g_i$ are the generators of the $i$-th copy of $\mathbb{Z}_2$.
\item $D_\infty:$ $\Gamma_2$, also known as the infinite dihedral group.
\item $\tau=\tau_{P_1,\dots,P_m}:$ The unitary representation defined by Proposition~\ref{fact:correspondance}, for orthogonal projections $P_1,\dots,P_m$ on a Hilbert space $V$.
\item $\operatorname{id}_V:$ The identity map on a Hilbert space $V$.
\item $E_n:$ The identity matrix on $\mathbb{C}^n$.
\item $\mathbb{C}\langle x_1,\dots,x_m \rangle:$ The set of all non-commutative polynomials over $\mathbb{C}$
with non-commutative variables $x_1,\dots,x_m$.
\item $\displaystyle \bigoplus_{i \in I}T_i:$ The operator on a direct sum $\displaystyle\bigoplus_{i \in I} V_i$ of Hilbert spaces $\{V_i\}_{i \in I}$
for a family $\{T_i\}_{i \in I}$ of bounded linear operators where each $T_i$ acts on $V_i$, the action is defined by
\begin{equation*}
\bigoplus_{i \in I}T_i \sum_{i \in I} v_i=\sum_{i \in I} T_i v_i\quad(i \in I, v_i \in V_i).
\end{equation*}
\item $\sigma(T):$ The spectrum for a bounded linear operator $T$ on a Hilbert space $V$.
\item $\sigma_p(T):$ The set of all eigenvalues of a bounded linear operator $T$ on a Hilbert space $V$.
\item $\rho(T):$ The spectral radius of a bounded normal operator $T$ on a Hilbert space $V$. $\rho(T)$ is defined by
\begin{equation*}
\rho(T)=\sup_{v \in V, \|v\|=1}|(Tv,v)|.
\end{equation*}
\item $\displaystyle \bigotimes_{i=1}^mV_i=V_1 \otimes \cdots \otimes V_m:$ The Hilbert tensor product of separable Hilbert spaces $V_1,\cdots,V_m$.
\item $P_i:$ The abbreviation for $\operatorname{id}_{V_1} \otimes \dots \otimes \operatorname{id}_{V_{i-1}} \otimes P_i \otimes \operatorname{id}_{V_{i+1}} \otimes \dots \otimes \operatorname{id}_{V_m}$ acting on a Hilbert tensor product $\displaystyle \bigotimes_{i=1}^mV_i$, where $P_i$ is an orthogonal projection on $V_i$.
\end{itemize}

\subsection{Unitary representations of $\mathbb{Z}_2 * \cdots * \mathbb{Z}_2$}
\label{subsec:unitaryrepresentation}

For a given Hilbert space $V$ and $m \in \mathbb{N}$,
there exists a one-to-one correspondence between $m$ orthogonal projections on $V$ and unitary representations of 
$\Gamma_m$ on $V$, as described below. For basic notions and properties of unitary representations of groups,
see, e.g., \cite{MSugiura1990}.

\begin{Proposition}
\label{fact:correspondance}
Let $P_1,\dots,P_m$ be $m$ orthogonal projections on $V$,
and define the representation $\tau=\tau_{P_1,\dots,P_m}$ of $\Gamma_m$ by
$\displaystyle\tau(g_i)=2P_i-\operatorname{id}_{V}\quad(i=1,\dots,m)$.
Then $(V,\tau)$ is a unitary representation of $\Gamma_m$.

Conversely, for a unitary representation $\tau$ of $\Gamma_m$ whose representation space is $V$,
define $P_1,\dots,P_m$ by
$\displaystyle P_i=\dfrac{\tau(g_i)+\operatorname{id}_{V}}{2}\quad(i=1,\dots,m)$.
Then $P_1,\dots,P_m$ are orthogonal projections on $V$.
\end{Proposition}

The following Proposition~\ref{fact:infinitedihedralgroupfinitedimensionalrepresentationdecomposition} is well known concerning the irreducible decomposition of finite dimensional unitary representations of 
$D_\infty$. In quantum information theory,
Proposition~\ref{fact:infinitedihedralgroupfinitedimensionalrepresentationdecomposition}
is referred to as Jordan's lemma.

\begin{Proposition}[Jordan's lemma]
\label{fact:infinitedihedralgroupfinitedimensionalrepresentationdecomposition}
Let
\begin{enumerate}
\item[(S1)] $V$ be a separable Hilbert space,
\item[(S2)] $P$ and $Q$ be a pair of orthogonal projections on $V$, and
\item[(S3)] $A$ and $B$ be defined by $A=2P-\operatorname{id}_V$ and $B=2Q-\operatorname{id}_V$,
\end{enumerate}
then
\begin{enumerate}
\item[(i)] If $\dim V <\infty$, there exists $M \in \mathbb{N}$ and an orthogonal direct sum decomposition
$\displaystyle V=\bigoplus_{i=1}^M W_i$ such that $W_i$ is an irreducible invariant subspace under $\tau_{P,Q}$ for any $i \in \{1,\dots,M\}$.
\item[(ii)] If $\tau_{P,Q}$ is an irreducible unitary representation of $D_\infty$, then $\dim V \le 2$.
\item[(iii)] If $\tau_{P,Q}$ is an irreducible unitary representation of $D_\infty$ and $\dim V =2$,
then there exists $x \in (0,1)$ such that with respect to a suitable orthonormal basis
\begin{align*}
2P-E_2=
\begin{pmatrix}
2x - 1 & 2\sqrt{x(1-x)}\\
2\sqrt{x(1-x)} & -(2x-1)
\end{pmatrix},
2Q-E_2=
\begin{pmatrix}
1 & 0\\
0 & -1
\end{pmatrix}.
\end{align*}
\end{enumerate}
\end{Proposition}
\begin{proof}[Proof of (i)]
This follows from the general theory of finite dimensional unitary representations of groups; see, e.g., \cite{MSugiura1990}.
\end{proof}
\begin{proof}[Proof of (ii)]
This was proven by Behncke~\cite{HBehncke1971}.
\end{proof}
\begin{proof}[Proof of (iii)]
This follows from Halmos' theorem~\cite{PRHalmos1969}.
\end{proof}

The following result
is a consequence of Proposition~\ref{fact:infinitedihedralgroupfinitedimensionalrepresentationdecomposition}
and
elementary calculations.

\begin{Proposition}
\label{fact:infinitedihedralgroupfinitedimensionalrepresentation}
Let
\begin{enumerate}
\item[(S1)] $V$ be a finite dimensional Hilbert space,
\item[(S2)] $P$ and $Q$ be a pair of orthogonal projections on $V$, and
\item[(S3)] $A,B$ be defined by $A=2P-\operatorname{id}_V$ and $B=2Q-\operatorname{id}_V$,
\end{enumerate}
then
\begin{enumerate}
\item[(i)] If $\dim V = 1$, then we have
\begin{equation*}
\sigma(A+B) \subset \{-2,0,2\}, \sigma([A,B])=\{0\}.
\end{equation*}
\item[(ii)] If $\dim V=2$ and $\tau_{P,Q}$ is an irreducible unitary representation of $D_\infty$, then there exists $x \in (0,1)$ such that
\begin{align*}
\sigma(A+B)&=\{-2\sqrt{x}, 2\sqrt{x}\},\sigma(A-B)=\{-2\sqrt{1-x}, 2\sqrt{1-x}\},\\
\sigma([A,B])&=\{-4\sqrt{x(1 - x)}i, 4\sqrt{x(1 - x)}i\}.
\end{align*}
Define $u(x), v(x)$ by
\begin{align*}
u(x)&=
\dfrac{1}{\sqrt{\dfrac{1-x}{(1+\sqrt{x})^2}+1}}
\begin{pmatrix}
-\dfrac{\sqrt{1-x}}{1+\sqrt{x}}\\
1
\end{pmatrix},\\
v(x)&=
\dfrac{1}{\sqrt{\dfrac{1-x}{(1-\sqrt{x})^2}+1}}
\begin{pmatrix}
\dfrac{\sqrt{1-x}}{1-\sqrt{x}}\\
1
\end{pmatrix}.
\end{align*}
Then $u(x), v(x)$ are unit eigenvectors of $A+B$,
corresponding to\\ $-2\sqrt{x},2\sqrt{x}$, respectively.
\item[(iii)]
For any orthogonal projections $P,Q$ in a two dimensional
Hilbert space $V$
such that $\tau_{P,Q}$ is an irreducible unitary representation of $D_\infty$,
\begin{equation*}
\rho([A,B])=|([A,B](\alpha u(x)+\beta v(x)),\alpha u(x)+\beta v(x))|,
\end{equation*}
for some $\alpha, \beta \in \mathbb{C}\text{ such that }|\alpha|^2+|\beta|^2=1$.
\item[(iv)] Let $\lambda$ be an eigenvalue of $A+B$ such that $\lambda^2$ is closest to $2$. Then
\begin{align*}
\sigma([A,B])&=b(\lambda).
\end{align*}
\end{enumerate}
\end{Proposition}

\subsection{Bell-CHSH operator}
\label{subsec:chshoperatorintroduction}

\subsubsection{Spectral radius of Bell-CHSH operator}

We define the Bell-CHSH operator and briefly summarize various results on it.
For two separable Hilbert spaces $V_1$ and $V_2$, and
a pair of orthogonal projections $P_{i,1}$ and $P_{i,2}$ on $V_i$ for $i=1,2$,
define $A_j=2P_{1,j}-\operatorname{id}_{V_1}$ and $B_j=2P_{2,j}-\operatorname{id}_{V_2}$ for $j=1,2$.
The Bell-CHSH operator $\mathcal{B}(P_{1,1},P_{1,2},P_{2,1},P_{2,2})$ on $V_1 \otimes V_2$ is defined as follows:
\begin{equation*}
\mathcal{B}(P_{1,1},P_{1,2},P_{2,1},P_{2,2})=(A_{1}+A_{2}) \otimes B_{1} + (A_{1}-A_{2}) \otimes B_{2}.
\end{equation*}
We denote $\mathcal{B}(P_{1,1},P_{1,2},P_{2,1},P_{2,2})$
by $\mathcal{B}$ whenever no confusion arises.

The spectral radius of $\mathcal{B}(P_{1,1},P_{1,2},P_{2,1},P_{2,2})$ is given by the following formula as established in \cite{LAKhalfin1985, LJLandau1987},
\begin{equation}
\label{eq:khalfintsirelsonlandau}
\rho(\mathcal{B}(P_{1,1},P_{1,2},P_{2,1},P_{2,2}))=\sqrt{4+\rho([A_1,A_2])\rho([B_1,B_2])}.
\end{equation}

This formula shows that the problem of estimating $\rho(\mathcal{B}(P_{1,1},P_{1,2},P_{2,1},P_{2,2}))$ is equivalent to that of estimating $\rho([A_1,A_2])$ and $\rho([B_1,B_2])$.
Remark that
\begin{equation*}
[A_1,A_2]=4[P_{1,1},P_{1,2}], [B_1,B_2]=4[P_{2,1},P_{2,2}].
\end{equation*}

\subsubsection{Tsirelson's inequality}

Tsirelson's inequality is stated as Theorem~\ref{theorem:tsirelsoninequality}.

\begin{Theorem}[Tsirelson's inequality \cite{BSTsirelson1980}]
\label{theorem:tsirelsoninequality}
Let
\begin{enumerate}
\item[(S1)] $V_1$ and $V_2$ be separable Hilbert spaces,
\item[(S2)] $P_{i,1}$ and $P_{i,2}$ be orthogonal projections on $V_i$ for $i \in \{1,2\}$,
\end{enumerate}
then
\begin{equation}
\label{eq:tsirelsoninequality}
\rho\left(\mathcal{B}\left(P_{1,1},P_{1,2},P_{2,1},P_{2,2}\right)\right) \le 2\sqrt{2}.
\end{equation}
\end{Theorem}

\subsection{Tridiagonal matrices}
\label{subsec:tridiagonalmatrixintroduction}

\ifthenelse{\value{arxiv} = 1}{
\subsubsection{Sufficient conditions for distinct eigenvalues}
}

We apply the following well-known fact to estimate
 the spectral radius of $[A,B]$; see, e.g., \cite{TSChihara1978} for the details.

\begin{Proposition}
\label{fact:tridiagonaleigenvaluesdistinctcondition}
Let $C$ be a symmetric tridiagonal matrix defined by
\begin{equation*}
C=
\begin{pmatrix}
c_{1,1} & c_{1,2} &           0 &           0 & \dots &  0 & 0 & 0\\
c_{2,1} & c_{2,2} & c_{2,3} &           0 & \dots &  0 & 0 & 0\\
          0 & c_{3,2} & c_{3,3} & c_{3,4} & \dots &  0 & 0 & 0\\
    \dots &     \dots &      \dots &    \dots & \dots & \dots & \dots & \dots\\
          0 &            0 &           0 &           0 & \dots & c_{n-1,n-2} & c_{n-1,n-1} & c_{n-1,n}\\
          0 &            0 &           0 &           0 & \dots & 0 & c_{n,n-1} & c_{n,n}
\end{pmatrix}
\end{equation*}
such that
\begin{equation*}
c_{i,i+1} = c_{i+1,i} \ne 0\quad (i=1,2,\dots,n-1).
\end{equation*}
Then the eigenvalues of $C$ are distinct.
\end{Proposition}

\ifthenelse{\value{arxiv} = 1}{
\subsubsection{Explicit eigenvalues and eigenvectors of Toeplitz matrices}
\label{subsubsec:toeplitz}

Yueh and Cheng \cite{WCYueh2008} provides a method for constructing eigenpairs
 of the following $n \times n$ tridiagonal matrix:
\begin{equation*}
X
=
\begin{pmatrix}
    b + \gamma &  c & 0 & \dots & 0 & \alpha\\
                     a &   b & c & \dots & 0 & 0\\
                     0 &  a &  b & \dots & 0 & 0\\
                    \dots & \dots & \dots & \dots & \dots & \dots\\
                     0 & 0 & 0 & \dots & b & c\\
               \beta & 0 & 0 & \dots & a & b + \delta\\
\end{pmatrix}
\end{equation*}
where $n \in \mathbb{N}$ and $a,b,\gamma, \delta, \alpha, \beta \in \mathbb{C}$.
We focus on the case when $a=c \in \mathbb{R}$ and $b \in \mathbb{R}$.
Define $\mu$ by $\mu=\sqrt{\dfrac{a}{c}}$.
Assume that we find $\phi \in [0,2\pi)$ such that
 $\sin\phi \ne 0$ and $\phi$ satisfies the following equation:
\begin{align*}
&\mu^n\left(ac\sin\left(\left(n+1\right)\phi\right)+\left(\gamma \delta - \alpha \beta\right)\sin\left(\left(n-1\right)\phi\right)-c\mu\left(\gamma+\delta\right)\sin\left(n\phi\right)\right)\\
&-\left(c \alpha \mu^{2n}+\alpha\beta\right)\sin\phi=0.
\end{align*}
Then according to Theorem~3.1 of \cite{WCYueh2008},
\begin{equation*}
\lambda=b + 2c\mu \cos\phi
\end{equation*}
is an eigenvalue $\lambda$ of $X$.
Moreover, from (3.3) of \cite{WCYueh2008},
for any eigenvector $u$ of $X$ corresponding to $\lambda$,
\begin{align}
\label{eq:yueh33}
u_j=&\dfrac{2i}{\sqrt{\omega}}\left\{cu_1\mu^j \sin(j\phi)-(\alpha u_n + \gamma u_1)\mu^{j-1}\sin((j-1)\phi)\right\}\notag\\
&-\dfrac{2i}{\sqrt{\omega}}H_j^n(\beta u_1 + \delta u_n)\mu^{j-n}\sin((j-n)\phi)\quad(j=1,2,\dots,n),
\end{align}
where
\begin{align*}
H_j^n&=
\left\{
\begin{array}{ll}
0 & \quad j < n, \\
1 & \quad j \ge n.
\end{array}
\right.
\end{align*}
}

\section{Finite dimensional approximation of orthogonal projections}
\label{sec:finitedimensionalapproximation}

We define finite dimensional approximations of a finite family of orthogonal projections,
and present their basic properties and applications.

In Section~\ref{subsec:finitedimensionalapproximationbasicproperties},
we define finite dimensional approximations of a finite family of orthogonal projections, and present their basic properties.
In particular, we prove that for any finite family of orthogonal projections
on a separable Hilbert space $V$ there exists a finite dimensional approximation.

In Section~\ref{subsec:reduction},
we present an application of finite dimensional approximations of orthogonal projections. We give a method to reduce inequalities involving orthogonal projections and continuous functions on a Hilbert space to the finite dimensional case.

In Section~\ref{subsec:tsirelson},
we give a new proof of Tsirelson's inequality via this reduction method.

\subsection{Basic properties of finite dimensional approximation}
\label{subsec:finitedimensionalapproximationbasicproperties}

The definition of finite dimensional approximation of orthogonal projections is given below.

\begin{Definition}
\label{def:finitedimensional_approximation}
Let
\begin{enumerate}
\item[(S1)] $V$ be a separable Hilbert space,
\item[(S2)] $P_1,\dots,P_m$ be a finite family of orthogonal projections on $V$.
\end{enumerate}
We say that a sequence $\left\{\left(V_j,\left\{P_{j,1},\dots,P_{j,m}\right\}\right)\right\}_{j=1}^\infty$
is a finite dimensional approximation of the pair $\left(V,\left\{P_1,\dots,P_m\right\}\right)$ if and only if the following conditions hold:
\begin{enumerate}
\item[(i)] For any $j \in \mathbb{N}$,
$V_j$ is a finite dimensional subspace of $V$ and
$\{P_{j,k}\}_{k=1}^m$ is a finite family of orthogonal projections on $V_j$.
\item[(ii)] For any $v \in V$, it follows that
\begin{equation*}
\lim_{M \to \infty}\max \{\|v-\pi_M v\|,\|P_1 v-P_{M,1} \pi_M v\|,\dots,\|P_m v-P_{M,m} \pi_M v\|\}=0,
\end{equation*}
where $\pi_M$ is the orthogonal projection from $V$ onto $V_M$.
\end{enumerate}
\end{Definition}

We present two equivalent characterizations of finite dimensional approximation of orthogonal projections in separable Hilbert spaces.

\begin{Proposition}
Assume the setups $(S1), (S2)$ in {\rm{Definition~\ref{def:finitedimensional_approximation}}},
and let
\begin{enumerate}
\item[(S3)] $\{(V_j,\{P_{j,1},\dots,P_{j,m}\})\}_{j=1}^\infty$ be a sequence such that $V_j$
is a finite dimensional subspace of $V$
and $P_{j,1},\dots,P_{j,m}$ are orthogonal projections on $V_j$ for any $j \in \mathbb{N}$,
\end{enumerate}
then the following conditions are equivalent:
\begin{enumerate}
\item[(i)] The sequence $\{(V_j,\{P_{j,1},\dots,P_{j,m}\})\}_{j=1}^\infty$
is a finite dimensional approximation of $(V,\{P_1,\dots,P_m\})$.
\item[(ii)] For some dense countable subset $D=\{v_i\}_{i=1}^\infty$ of $V$ and any $i \in \mathbb{N}$,
\begin{equation*}
\lim_{M \to \infty}\max\{\|v_i-\pi_M v_i\|,\|P_1 v_i-P_{M,1} \pi_M v_i\|,\dots,\|P_m v_i-P_{M,m} \pi_M v_i\|\}=0,
\end{equation*}
where $\pi_M$ is the orthogonal projection from $V$ onto $V_M$.
\item[(iii)] For any dense countable subset $D=\{v_i\}_{i=1}^\infty$ of $V$ and any $i \in \mathbb{N}$,
\begin{equation*}
\lim_{M \to \infty}\max\{\|v_i-\pi_M v_i\|,\|P_1 v_i-P_{M,1} \pi_M v_i\|,\dots,\|P_m v_i-P_{M,m} \pi_M v_i\|\}=0,
\end{equation*}
where $\pi_M$ is the orthogonal projection from $V$ onto $V_M$.
\end{enumerate}
\end{Proposition}
\begin{proof}
It is clear that (i) $\implies$ (iii) $\implies$ (ii).
We prove that (ii) $\implies$ (i).
Let $v \in V$ and $\varepsilon > 0$.
There exist $i \in \mathbb{N}$ and $n_0 \in \mathbb{N}$ such that
for any $M \in [n_0,\infty) \cap \mathbb{N}$
\begin{align*}
\max\{\|v-v_i\|, \|v_i-\pi_M v_i\|,\|P_1 v_i-P_{M,1} \pi_M v_i\|,\dots,\|P_m v_i-P_{M,m} \pi_M v_i\|\}<\dfrac{\varepsilon}{4}.
\end{align*}
Let $M \in [n_0,\infty) \cap \mathbb{N}$ and $k \in \{1,\dots,m\}$.
Since the operator norm of every orthogonal projection is at most $1$,
we have
\begin{align*}
&\|v - \pi_Mv\|=\|v_i-\pi_Mv_i+(\mathrm{id}_V - \pi_M)(v-v_i)\|<\varepsilon,\\
&\|P_k v-P_{M,k} \pi_M v\|=\|P_k v_i-P_{M,k}\pi_Mv_i+(P_k-P_{M,k}\pi_M)(v-v_i)\|<\varepsilon.
\end{align*}
\end{proof}

For any finite family of orthogonal projections
on a separable Hilbert space $V$, there exists a finite dimensional approximation.
This theorem will be applied in Section~\ref{subsec:reduction}.

\begin{Theorem}
\label{theorem:finite_dimensional_approximation_existence}
Assume the setups $(S1), (S2)$ in {\rm{Definition~\ref{def:finitedimensional_approximation}}}.
Then there exists a finite dimensional approximation of $(V,\{P_1,\dots,P_m\})$.
\end{Theorem}
\begin{proof}
Let $D=\{v_i\}_{i=1}^\infty$ be a dense countable subset of $V$.
Let $n \in \mathbb{N}$.
For $i \in \{1,\dots,m\}$, define $W_i$ by
$W_i=\langle P_i v_1,\dots,P_i v_n \rangle$
and let $Q_i$ be the orthogonal projection onto $W_i$.
Let $i \in \{1,\dots,m\}$ and $j \in \{1,\dots,n\}$.
Since
\begin{equation*}
v_j=P_iv_j+(\mathrm{id}_V-P_i)v_j, P_iv_j \in W_i, (\mathrm{id}_V-P_i)v_j \perp W_i,
\end{equation*}
we have $P_iv_j=Q_iv_j$.

Define $W$ by
$\displaystyle W=\sum_{i=1}^m W_i$.
It follows that $W$ is a finite dimensional invariant subspace
under $Q_1,\dots,Q_m$.
Therefore, there exists
an orthonormal basis $\{x_k\}_{k=1}^{\dim W}$
of $W$ and an orthonormal basis $\{y_l\}_{l=1}^{\infty}$
of $W^{\perp}$.
Since $V=W \oplus W^{\perp}$,
there exists a sufficiently large $L \in \mathbb{N}$ such that
$\|v_j-\pi_{L}v_j\| < \dfrac{1}{n}\quad(j=1,\dots,n)$,
where $\pi_L$ is the orthogonal projection
onto $\langle \{x_k\}_{k=1}^{\dim W} \cup \{y_l\}_{l=1}^{L} \rangle$.
Fix such $L$ and define $L(n)=L$.
Let $V_L$ be $\langle \{x_k\}_{k=1}^{\dim W} \cup \{y_l\}_{l=1}^{L} \rangle$.
Since $Q_i | W^{\perp}=0$ for any $i \in \{1,\dots,m\}$,
 $V_L$ is invariant under $Q_1,\dots,Q_m$.

Define $P_{L,1},\dots,P_{L,m}$ by
$P_{L,i}=Q_i|V_L\quad(i=1,\dots,m)$.
Then $V_L$ is an invariant subspace under $P_{L,1},\dots,P_{L,m}$.
Hence $P_{L,1},\dots,P_{L,m}$ are orthogonal projections
on $V_{L}$.
Moreover, for any $i \in \{1,\dots,m\}$, $j \in \{1,\dots,n\}$
\begin{equation*}
\|P_i v_j-P_{L,i}\pi_{L}v_j\|=\|Q_{i} (v_j-\pi_{L}v_j)\| < \dfrac{1}{n}.
\end{equation*}

We conclude that $\{(V_{L(n)}, \{P_{L(n),1},\dots,P_{L(n),m}\})\}_{n=1}^\infty$ is
a finite dimensional approximation of $(V,\{P_1,\dots,P_m\})$.
\end{proof}

Every finite dimensional approximation of orthogonal projections
also approximates every polynomial of the projections, as follows.
This proposition will also be applied in Section~\ref{subsec:reduction}.

\begin{Proposition}
\label{proposition:finite_dimensional_approximation_polynomial}
Assume the setups $(S1), (S2)$ in {\rm{Definition~\ref{def:finitedimensional_approximation}}}, and let
\begin{enumerate}
\item[(S3)] $\{(V_j,\{P_{j,1},\dots,P_{j,m}\})\}_{j=1}^\infty$ be a finite dimensional approximation  of \\$(V,\{P_1, \dots, P_m\})$, and
\item[(S4)] $f(x_1,\dots,x_m) \in \mathbb{C}\langle x_1,\dots,x_m \rangle$.
\end{enumerate}
Then, for any $v \in V$, we have
\begin{equation*}
\lim_{M \to \infty}\|f(P_1,\dots,P_m) v-f(P_{M,1},\dots,P_{M,m}) \pi_M v\|=0.
\end{equation*}
\end{Proposition}
\begin{proof}
Let $M=\deg(f)$.
Define the complex numbers $a_{i_1,\dots,i_k}$ \\$(k \in \{0,\dots,\deg(f)\}, i_1,\dots,i_k \in \{1,\dots,m\})$ by
\begin{equation*}
f(x_1,\dots,x_m)=\sum_{\substack{k \in \{0,\dots,\deg(f)\}, \\ i_1,\dots,i_k \in \{1,\dots,m\}}}a_{i_1,\dots,i_k}x_{i_1}\cdots x_{i_k}.
\end{equation*}
Define $a, B$ by
\begin{align*}
a&=\max\{|a_{i_1,\dots,i_k}|\ | k \in \{0,\dots,\deg(f)\},
i_1,\dots,i_k \in \{1,\dots,m\}\},\\
B&=(a+1)(m+1)^{M+1}.
\end{align*}
Let $v \in V$ and $\varepsilon > 0$.
Then
\begin{equation*}
\left\{P_{i_1}\cdots P_{i_k}v\ \middle|k \in \{0,\dots,M\}, i_1,\dots,i_k \in \{1,\dots,m\}\right\}
\end{equation*}
is a finite family of vectors.
Hence, for any sufficiently large $n \in \mathbb{N}$ and any $k \in \{0,\dots,M\}$ and any $i_1,\dots,i_k \in \{1,\dots,m\}$,
we have
\begin{align*}
\left\|P_{i_1}\cdots P_{i_k}v-P_{n,i_1}\cdots P_{n,i_k}\pi_{n}v\right\|<\dfrac{1}{B}\varepsilon.
\end{align*}
Therefore, for any sufficiently large $n \in \mathbb{N}$
\begin{equation*}
\left\|f(P_1,\dots,P_m) v-f(P_{n,1},\dots,P_{n,m}) \pi_n v\right\|<\varepsilon.
\end{equation*}
\end{proof}

The following proposition presents a method for constructing finite dimensional approximations of orthogonal projections $P_1,\dots,P_m$ on a separable Hilbert space $V$ via a decomposition into invariant subspaces under these projections.
This proposition will be applied in Section~\ref{sec:two}.

\begin{Proposition}
\label{proposition:finite_dimensional_approximation_directsum}
Assume the setups $(S1), (S2)$ in {\rm{Definition~\ref{def:finitedimensional_approximation}}},
and let
\begin{enumerate}
\item[(S3)] $\displaystyle V=\bigoplus_{i \in I} W_i$ be an orthogonal direct sum decomposition, where $I$ is an at most countable set,
and $W_i$ is an invariant subspace under $P_1,\dots,P_m$ for any $i \in I$,
\item[(A1)] $\{(W_{i,n},\{P_{i,n,1},\dots,P_{i,n,m}\})\}_{n=1}^\infty$
be a finite dimensional approximation of $(W_i, \{P_1|W_i,\dots,P_m|W_i\})$ for $i \in I$, and
 \item[(S4)] $\{I_i\}_{i=1}^\infty$ be a family of finite subsets of $I$ such that $I_i \subset I_{i+1}$ for any $i \in \mathbb{N}$ and $\displaystyle \bigcup_{i=1}^\infty I_i=I$,
\end{enumerate}
then
\begin{equation*}
\left\{\left(\bigoplus_{i \in I_n} W_{i,n}, \left\{\bigoplus_{i \in I_n} P_{i,n,1},\dots,\bigoplus_{i \in I_n} P_{i,n,m}\right\}\right)\right\}_{n=1}^\infty
\end{equation*}
is a finite dimensional approximation of $(V,\{P_1,\dots,P_m\})$.
\end{Proposition}
\begin{proof}
Let $v \in V$ and $\varepsilon > 0$.
Then there exists $n_1 \in \mathbb{N}$ such that
for any $n \in [n_1,\infty) \cap \mathbb{N}$
\begin{equation*}
\left\|v-\sum_{j \in I_{n}}w_j\right\|<\dfrac{\varepsilon}{8},
\end{equation*}
where $w_j=\pi_{W_j}v$ and $\pi_{W_j}$ is the orthogonal projection onto $W_j$ for $j \in I_{n}$.
For $n \in \mathbb{N}$, let $\pi_n$ be $\displaystyle \bigoplus_{j \in I_n} \pi_{W_j}$.

Since $I_{n_1}$ is a finite set, there exists $n_2 \in \mathbb{N}$ such that
\begin{equation*}
\left\|P_{j,k}w_j-P_{j,n,k}\pi_{j,n} w_j\right\|<\dfrac{\varepsilon}{4\#I_{n_1}}
\end{equation*}
for any $j \in I_{n_1}$, $k \in \{1,\dots,m\}$, and $n \in [n_2,\infty) \cap \mathbb{N}$.
Here $\pi_{j,n}$ is the orthogonal projection onto $W_{j,n}$.
For $n,n' \in \mathbb{N}$, let $\pi_{n,n'}$ be $\displaystyle \bigoplus_{j \in I_n} \pi_{W_{j,n'}}$.

Let $k \in \{1,\dots,m\}$ and $n \in [\max\{n_1,n_2\},\infty) \cap \mathbb{N}$.
From the definition of $n_1$, we have
$
\|v-\pi_n v\|<\dfrac{\varepsilon}{8}
$. It follows that
$
\|\pi_n v - \pi_{n_1} v\|<\dfrac{\varepsilon}{4}
$.
From the definition of $n_2$, we have
$
\displaystyle\left\|\bigoplus_{j \in I_{n_1}}P_{j,k}\pi_{n_1}v-\bigoplus_{j \in I_{n_1}}P_{j,n,k}\pi_{n_1,n} v\right\|<\dfrac{\varepsilon}{4}
$.
We conclude that
\begin{equation*}
\left\|P_{k}v-\bigoplus_{j \in I_{n}}P_{j,n,k} \pi_{n,n} v  \right\|<\varepsilon.
\end{equation*}
\end{proof}

It is possible to construct a finite dimensional approximation of a certain type of orthogonal projections
on Hilbert tensor products, as follows.
This proposition will be applied in Section~\ref{subsec:reduction}.

\begin{Proposition}
\label{proposition:finite_dimensional_approximation_polynomial_tensor}
Let
\begin{enumerate}
\item[(S1)] $V_1,\dots,V_m$ be a finite family of separable Hilbert spaces,
\item[(S2)] $P_{i,1}, \dots, P_{i,n_i}$ be a finite family of orthogonal projections  on $V_i$ for $i \in \{1,\dots,m\}$, and
\item[(S3)] $\{(V_{i,n},\{P_{i,n,1}, \dots, P_{i,n,n_i}\})\}_{n=1}^\infty$ be a finite dimensional approximation  of $(V_i,\{P_{i,1}, \dots, P_{i,n_i}\})$ for $i \in \{1,\dots,m\}$,
\end{enumerate}
then
\begin{equation*}
\left\{\left(\bigotimes_{l=1}^m V_{l,n},\left\{P_{i,n,k_i} \middle| k_i \in [1,n_i] \cap \mathbb{N}\quad (i=1,\dots,m)\right\}\right)\right\}_{n=1}^\infty
\end{equation*}
is a finite dimensional approximation of
\begin{equation*}
\left(\bigotimes_{l=1}^m V_{l},\left\{P_{i,k_i} \middle| k_i \in [1,n_i] \cap \mathbb{N}\quad (i=1,\dots,m)\right\}\right).
\end{equation*}
\end{Proposition}
\begin{proof}
Let $\displaystyle \psi \in \bigotimes_{i=1}^m V_i$ and $\varepsilon>0$.
From the definition of $\displaystyle \bigotimes_{i=1}^m V_i$, the linear span
$\langle \{v_1 \otimes \dots \otimes v_m|v_i \in V_i\quad(i=1,\dots,m)\} \rangle$
is dense in $\displaystyle \bigotimes_{i=1}^m V_i$.
Therefore, there exist $M \in \mathbb{N}$ and
$v_{i,1},\dots,v_{i,M} \in V_i$ $(i=1,\dots,m)$
such that
$\displaystyle \left\|\psi - \sum_{k=1}^M v_{1,k} \otimes \dots \otimes v_{m,k}\right\|<\dfrac{\varepsilon}{3}$.
Define $L$ by 
\begin{equation*}
L=\max\left\{\|v_{i,k}\|\ |i=1,\dots,m,k=1,\dots,M\right\}+1.
\end{equation*}
There exists $n_0 \in \mathbb{N}$ such that
for any $n \in [n_0,\infty) \cap \mathbb{N}$, $i \in \{1,\dots,m\}$, $j \in \{1,\dots,n_i\}$, and $k \in \{1,\dots,M\}$, we have
\begin{equation*}
\left\|P_{i,j}v_{i,k}-P_{i,n,j}\pi_{i,n}v_{i,k}\right\|<\dfrac{\varepsilon}{3ML^m}.
\end{equation*}
Here, $\pi_{i,n}$ is the orthogonal projection onto $V_{i,n}$ for $i \in \{1,\dots,m\}$ and $n \in \mathbb{N}$.
Let $n \in [n_0,\infty) \cap \mathbb{N}$.
Then, we have
\begin{equation*}
\left\|P_{i,j_i}\sum_{k=1}^M\bigotimes_{l=1}^m v_{l,k} - P_{i,n,j_i}\bigotimes_{l=1}^m\pi_{l,n}\sum_{k=1}^M\bigotimes_{l=1}^m v_{l,k}\right\|
<\dfrac{\varepsilon}{3}
\end{equation*}
for any $i \in \{1,\dots,m\}$, $j_i \in \{1,\dots,n_i\}$.
Note that
$\displaystyle\bigotimes_{l=1}^m \pi_{l,n}$ is the orthogonal projection onto $\displaystyle\bigotimes_{l=1}^m V_{l,n}$.
We conclude that for any $i \in \{1,\dots,m\}$, $j_i \in \{1,\dots,n_i\}$,
\begin{equation*}
\left\|P_{i,j_i}\psi-P_{i,n,j_i}\bigotimes_{l=1}^m \pi_{l,n}\psi\right\|
<\varepsilon.
\end{equation*}
\end{proof}

For a polynomial $f(P_1,\dots,P_m)$ of orthogonal projections $P_1,\dots,P_m$
on a Hilbert space $V$,
we define the finite dimensional
approximate point spectrum of $f(P_1,\dots,P_m)$
corresponding to a finite dimensional approximation
$\{(V_n,\{P_{n,1},\dots,P_{n,m}\})\}_{n=1}^\infty$, as follows.
Such a spectrum consists of all complex numbers $\lambda$ satisfying the following conditions:
\begin{itemize}
\item $\lambda$ is approximated by eigenvalues $\lambda_n$ $(n \in \mathbb{N})$ of $f(P_{n,1},\dots,P_{n,m})$.
\item The actions of $P_1,\dots,P_m$ on an eigenvector corresponding to $\lambda_n$ are approximated by $P_{n,1},\dots,P_{n,m}$.
\end{itemize}
This notion plays a role in Section~\ref{sec:chsh}.

\begin{Definition}
\label{def:finitedimensionalapproximatespectrum}
Assume the setups $(S1), (S2)$ in \rm{Definition~\ref{def:finitedimensional_approximation}},
and let
\begin{enumerate}
\item[(S3)] $f \in \mathbb{C}\langle x_1,\dots,x_m \rangle$,
\item[(S4)] $\{(V_n,\{P_{n,1},\dots,P_{n,m}\})\}_{n=1}^\infty$ be a finite dimensional approximation
of $(V,\{P_1,\dots,P_m\})$.
\end{enumerate}
We say $\lambda \in \mathbb{C}$ is in the finite dimensional
approximate point spectrum of $f(P_1,\dots,P_m)$
corresponding to $\{(V_n,\{P_{n,1},\dots,P_{n,m}\})\}_{n=1}^\infty$
if there exists a sequence $\{\lambda_n\}_{n=1}^\infty$ in $\mathbb{C}$ and a sequence $\{u_n\}_{n=1}^\infty$ of unit
vectors in $V$ such that
\begin{enumerate}
\item[(i)] For any $n \in \mathbb{N}$, $\lambda_n$ is an eigenvalue of $f(P_{n,1},\dots,P_{n,m})$.
\item[(ii)] For any $n \in \mathbb{N}$, $u_n$ is an eigenvector of $f(P_{n,1},\dots,P_{n,m})$
corresponding to $\lambda_n$.
\item[(iii)] $\displaystyle \lim_{n\to \infty}\lambda_n=\lambda$.
\item[(iv)] For any $j \in \{1,\dots,m\}$,
$\displaystyle \lim_{n\to \infty}\|P_j u_n-P_{n,j} u_n\|=0$.
\end{enumerate}
\end{Definition}

We note that every complex number in the finite dimensional approximate point spectrum of $f(P_1,\dots,P_m)$ lies in the approximate point spectrum of $f(P_1,\dots,P_m)$.
For the definition of the approximate point spectrum,
see, e.g., \cite{KYoshida1980}.

\subsection{Reduction of inequalities to the finite dimensional case}
\label{subsec:reduction}

In various fields, including quantum mechanics,
one encounters the following problem:
for a Hilbert space $V$,
a non-commutative polynomial $f$ of orthogonal projections
$P_1,\dots,P_m$,
and a real-valued continuous function $F$ on $V$,
estimate the following quantity,
\begin{align*}
\sup_{\substack{\Psi \in V, \|\Psi\|=1}}F(f(P_{1},\dots,P_{m})\Psi),
\end{align*}
where $F$ represents the quantity to be estimated.

In quantum information theory,
one often considers problems where $V$ is a Hilbert tensor product $\displaystyle\bigotimes_{i=1}^m V_i$,
and orthogonal projections are
$\operatorname{id}_{V_1} \otimes \dots \otimes P_{i,j}\otimes \dots \otimes \operatorname{id}_{V_m}$ ($i \in \{1,\dots,m\}$, $j \in \{1,\dots,n\}$).

As an example, consider the problem of estimating the spectral radius of the Bell-CHSH operator.
In this problem, we have $m = 2$, and $f$ and $F$ are given as follows.

\begin{Example}
\label{example:tsirelson}
Let
\begin{enumerate}
\item[(S1)] $V_1$, $V_2$ be separable Hilbert spaces,
\item[(S2)] $P_{i,j}$ be an orthogonal projection on $V_i$ for $i,j \in \{1,2\}$,
\item[(S3)] the non-commutative polynomial $f$ be defined by
\begin{equation*}
f(x_{1,1},\dots,x_{2,2})=2(x_{1,1}+x_{1,2}-1)(2x_{2,1}-1)+2(x_{1,1}-x_{1,2})(2x_{2,2}-1), \text{ and}
\end{equation*}
\item[(S4)] the real-valued continuous function $F$ on $V_1 \otimes V_2$ be defined by
\begin{equation*}
F(\Psi)=\|\Psi\|\quad (\Psi \in V_1 \otimes V_2),
\end{equation*}
\end{enumerate}
then the following identity holds:
\begin{equation*}
\sup_{\Psi \in V_1 \otimes V_2, \|\Psi\|=1}F(f(P_{1,1},P_{1,2},P_{2,1},P_{2,2})\Psi)=\rho(\mathcal{B}(P_{1,1},P_{1,2},P_{2,1},P_{2,2})).
\end{equation*}
\end{Example}
\begin{proof}
Define $A_{1},A_{2},B_{1},B_{2}$ by
$A_{i}=2P_{1,i}-\operatorname{id}_{V_1}$,
$B_{i}=2P_{2,i}-\operatorname{id}_{V_2}$
 $(i=1,2)$.
Then
\begin{align*}
&\quad\mathcal{B}
=
(A_1+A_2) \otimes B_1
+
(A_1-A_2) \otimes B_2\\
&=
2(P_{1,1}+P_{1,2}-\operatorname{id}_{V_1}) \otimes (2P_{2,1}-\operatorname{id}_{V_2})
+
2(P_{1,1}-P_{1,2}) \otimes (2P_{2,2}-\operatorname{id}_{V_2})\\
&=f(P_{1,1},P_{1,2},P_{2,1},P_{2,2}).
\end{align*}
Since $\mathcal{B}$ is a self-adjoint operator,
its spectral radius coincides with its operator norm.
Therefore, we have
\begin{align*}
&\quad\rho(\mathcal{B})
=
\|\mathcal{B}\|=\sup_{\substack{\Psi \in V_1 \otimes V_2,\\\|\Psi\|=1}}\|\mathcal{B}\Psi\|
=\sup_{\substack{\Psi \in V_1 \otimes V_2,\\\|\Psi\|=1}}F(f(P_{1,1},P_{1,2},P_{2,1},P_{2,2})\Psi).
\end{align*}
\end{proof}

The following Theorem~\ref{theorem:finite_dimensional_reduction} allows us to reduce the above problem to the case where the target Hilbert spaces are finite dimensional.
As an application, we present a new proof of Tsirelson's inequality
in Section \ref{subsec:tsirelson}.
Another application is a security proof of device-independent quantum
key distribution against a class of attacks known as collective attacks, which will be presented in our forthcoming work~\cite{TTsurumaru2025}.

\begin{Theorem}
\label{theorem:finite_dimensional_reduction}
Assume the setups $(S1), (S2)$ in \rm{Definition~\ref{def:finitedimensional_approximation}},
and let
\begin{enumerate}
\item[(S3)] $F$ be a real-valued continuous function on $V$,
\item[(S4)] $f(x_{1},\dots,x_{m}) \in \mathbb{C}\langle x_1,\dots,x_m\rangle$, and
\item[(A1)] there exists a constant $C \in \mathbb{R}$ such that
for any finite dimensional subspace $V_{f}$ of $V$, any orthogonal projections $P_{f,1},\dots,P_{f,m}$ on $V_{f}$,
and any unit vector $\displaystyle \Psi_f \in V_f$, it follows that
\begin{equation*}
F\left(f\left(P_{f,1},\dots,P_{f,m}\right) \Psi_f \right) \le C,
\end{equation*}
\end{enumerate}
then
for any unit vector $\displaystyle \Psi \in \bigotimes_{i=1}^m V_{i}$, we have
\begin{equation*}
F\left(f\left(P_{1},\dots,P_{m}\right)\Psi\right) \le C.
\end{equation*}
\end{Theorem}
\begin{proof}
From Theorem~\ref{theorem:finite_dimensional_approximation_existence},
there exists a finite dimensional approximation $\{(V_{n},\{P_{n,1},\dots,P_{n,m}\})\}_{n=1}^\infty$
of $(V,\{P_{1},\dots,P_m\})$.
Let $\Psi$ be a unit vector in $V$.
From Proposition~\ref{proposition:finite_dimensional_approximation_polynomial}, we have
\begin{align*}
\lim_{n \to \infty}\left\|f\left(P_{1},\dots,P_{m}\right)\Psi
-
f\left(P_{n,1},\dots,P_{n,m}\right)\dfrac{\pi_n\Psi}{\|\pi_n\Psi\|}\right\|=0,
\end{align*}
where $\pi_n$ is the orthogonal projection onto $V_n$.
From assumption (A1), for any $n \in \mathbb{N}$, it follows that
\begin{equation*}
F\left(f\left(P_{n,1},\dots,P_{n,m}\right)\dfrac{\pi_n\Psi}{\|\pi_n\Psi\|}\right) \le C.
\end{equation*}
Since $F$ is continuous, we conclude that
\begin{equation*}
F\left(f\left(P_{1},\dots,P_{m}\right)\Psi\right) \le C.
\end{equation*}
\end{proof}

From Theorem~\ref{theorem:finite_dimensional_reduction} and
Proposition~\ref{proposition:finite_dimensional_approximation_polynomial_tensor},
we have the following corollary.

\begin{Corollary}
\label{corollary:finite_dimensional_reduction}
Let
\begin{enumerate}
\item[(S1)] $V_1,\dots,V_m$ be a finite family of separable Hilbert spaces,
\item[(S2)] $P_{i,1}, \dots, P_{i,n}$ be a finite family of orthogonal projections on $V_i$ for $i \in \{1,\dots,m\}$,
\item[(S3)] $F$ be a real-valued continuous function on $\displaystyle \bigotimes_{i=1}^m V_i$,
\item[(S4)] $f(x_{1,1},\dots,x_{i,j},\dots,x_{m,n}) \in \mathbb{C}\langle x_{1,1},\dots,x_{i,j},\dots,x_{m,n}\rangle$, and
\item[(A1)] there exists a constant $C \in \mathbb{R}$ such that for any $i \in \{1,\dots,m\}$, any finite dimensional subspace $V_{f,i}$ of $V_i$, any orthogonal projections $P_{f,i,1},\dots,P_{f,i,n}$ on $V_{f,i}$,
and any unit vector $\displaystyle \Psi_f \in \bigotimes_{i=1}^m V_{f,i}$, it follows that
\begin{equation*}
F\left(f\left(P_{f,1,1},\dots,P_{f,i,j},\dots,P_{f,m,n}\right) \Psi_f \right) \le C,
\end{equation*}
\end{enumerate}
then
for any unit vector $\displaystyle \Psi \in \bigotimes_{i=1}^m V_{i}$, we have
\begin{equation*}
F\left(f\left(P_{1,1},\dots,P_{i,j},\dots,P_{m,n}\right)\Psi\right) \le C.
\end{equation*}
\end{Corollary}

\subsection{Application: New proof of Tsirelson's inequality}
\label{subsec:tsirelson}

First, note that Tsirelson's inequality
(\ref{eq:tsirelsoninequality}) is straightforward to prove
in the case where either subsystem of the bipartite quantum system has dimension at most two.

\begin{Lemma}
\label{lemma:tsirelsoninequalitytwodimensional}
We assume the setups and conditions of \rm{Theorem~\ref{theorem:tsirelsoninequality}},
and suppose that $\dim V_1 \le 2$.
Then inequality \rm{(\ref{eq:tsirelsoninequality})} holds.
\end{Lemma}
\begin{proof}
If $\dim V_1=1$,  (\ref{eq:tsirelsoninequality}) follows by simple calculus.
Thus, we can assume that $\dim V_1=2$ and $\tau_{P_{1,1},P_{1,2}}$ is irreducible.
From Proposition~\ref{fact:infinitedihedralgroupfinitedimensionalrepresentationdecomposition} and
Proposition~\ref{fact:infinitedihedralgroupfinitedimensionalrepresentation}, there exists $x \in (0,1)$ such that
\begin{align*}
&\mathcal{B}(P_{1,1},P_{1,2},P_{2,1},P_{2,2})\\
&=
\begin{pmatrix}
2x & 2\sqrt{x(1-x)}\\
2\sqrt{x(1-x)} & -2x
\end{pmatrix}
\otimes
B_{1}
+
\begin{pmatrix}
-2(1-x) & 2\sqrt{x(1-x)}\\
2\sqrt{x(1-x)} & 2(1-x)
\end{pmatrix}
\otimes
B_2.
\end{align*}
Define $C,D$ by
\begin{equation*}
C=
\begin{pmatrix}
2x & 2\sqrt{x(1-x)}\\
2\sqrt{x(1-x)} & -2x
\end{pmatrix},
D=
\begin{pmatrix}
-2(1-x) & 2\sqrt{x(1-x)}\\
2\sqrt{x(1-x)} & 2(1-x)
\end{pmatrix}.
\end{equation*}

From Proposition~\ref{fact:infinitedihedralgroupfinitedimensionalrepresentation},
the eigenvalues of $C$ are $-2\sqrt{x},2\sqrt{x}$
and the eigenvalues of $D$ are $-2\sqrt{1-x},2\sqrt{1-x}$.
Therefore, the spectral radius $\rho(\mathcal{B})$ is at most
$2\sqrt{x}+2\sqrt{(1-x)}$.
From simple differential calculus, we conclude that
\begin{equation*}
[0,1] \ni x \mapsto 2\sqrt{x}+2\sqrt{(1-x)} \in \mathbb{R}
\end{equation*}
takes the maximum value at $x=1/2$
and the maximum value is $2\sqrt{2}$.
Therefore, inequality (\ref{eq:tsirelsoninequality}) holds.
\end{proof}

Our new proof of Tsirelson's inequality is the following.

\begin{proof}[New proof of Tsirelson's inequality]

From Example~\ref{example:tsirelson}
and Corollary~\ref{corollary:finite_dimensional_reduction},
we may assume that $\dim V_1<\infty$.
Moreover, from
Proposition~\ref{fact:infinitedihedralgroupfinitedimensionalrepresentationdecomposition},
there exists a finite family $\{W_i\}_{i=1}^m$ of invariant subspaces under $P_{1,1}$ and $P_{1,2}$
such that $\dim W_i \le 2$ for any $i \in \{1,\dots,m\}$ and $\displaystyle V_1=\bigoplus_{i=1}^m W_i$ and
\begin{equation*}
W_{i_1} \perp W_{i_2}\quad (\text{for any }i_1, i_2 \in \{1,\dots,m\}\text{ such that } i_1\ne i_2 ).
\end{equation*}
Since
\begin{itemize}
\item $W_i \otimes V_2$ is invariant under $\mathcal{B}$ for any $i \in \{1,\dots,m\}$,
\item $\displaystyle V_1\otimes V_2=\bigoplus_{i=1}^m W_i \otimes V_2$,
\item 
$
W_{i_1} \otimes V_2 \perp W_{i_2} \otimes V_2\quad
$ for any $i_1, i_2 \in \{1,\dots,m\}$ such that $i_1\ne i_2$,
\end{itemize}

we may assume that $\dim V_1 \le 2$.
From Lemma~\ref{lemma:tsirelsoninequalitytwodimensional},
we conclude that inequality (\ref{eq:tsirelsoninequality}) holds.
\end{proof}

We remark on the difference between our proof and others.
There are many proofs of Tsirelson's inequality.
Several of such proofs are reviewed in \cite{LAKhalfin1992}.
The proof based on (\ref{eq:khalfintsirelsonlandau})
 is known to be particularly simple.

Tsirelson \cite{BSTsirelson1987} established an ingenious technique using the Clifford algebra
to reduce Tsirelson's inequality to the
case where $\dim V_i \le 2$ $(i=1,2)$ in
Theorem~\ref{theorem:tsirelsoninequality}.
Our proof presents another approach
that reduces Tsirelson's inequality to
the two dimensional case,
using Corollary~\ref{corollary:finite_dimensional_reduction}
and Proposition~\ref{fact:infinitedihedralgroupfinitedimensionalrepresentationdecomposition}.

\section{One-shifted form of two orthogonal projections}
\label{sec:two}

We introduce one-shifted form matrix representations for pairs of orthogonal projections on infinite dimensional Hilbert spaces, and present their basic properties and applications.
Let $V$ be a separable infinite dimensional Hilbert space, and $P, Q$ be orthogonal projections on $V$.

In Section~\ref{subsec:oneshiftedformbasic},
we define one-shifted form matrix representations and present their basic properties.
In particular, we prove that if $P$ and $Q$ admit a one-shifted form
then they admit an explicit finite dimensional approximation.

In Section~\ref{subsec:jordan},
we present a complete orthogonal direct sum decomposition of $V$ into invariant subspaces
$\{V_i\}_{i \in I}$
under $P$ and $Q$ such that for any $i \in I$ $\dim(V_i) \le 2$, or 
$P|V_i$ and $Q|V_i$ admit a one-shifted form.
This decomposition is an extension of Jordan's lemma
for finite dimensional Hilbert spaces.

In Section~\ref{subsec:explicit},
we present an explicit formula for finite dimensional approximation of pairs of orthogonal projections, based on the extension of Jordan's lemma.

\subsection{Basic properties of one-shifted form}
\label{subsec:oneshiftedformbasic}

We define one-shifted form matrix representations for pairs of orthogonal projections
as follows.

\begin{Definition}[One-shifted form]
\label{def:oneshiftedform}
Let
\begin{enumerate}
\item[(S1)] $V$ be an infinite dimensional Hilbert space,
\item[(S2)] $P$ and $Q$ be orthogonal projections on $V$, and
\item[(S3)] $\mathcal{U}=\{u_1,v_1,u_2,v_2,\dots\}$ be an orthonormal basis of $V$.
\end{enumerate}
We say that the pair $(V,\{P,Q\})$ admits a one-shifted form
in $\mathcal{U}$
if and only if there are two sequences $\Theta=\{\theta_i\}_{i=1}^\infty, \Omega=\{\omega_i\}_{i=1}^\infty$ of angles in $(0,\pi)$ such that the following conditions hold:
\begin{enumerate}
\item[(i)] For any $i \in \mathbb{N}$,
the spanned subspace $\langle u_i, v_i \rangle$ is invariant under $P$.
The matrix representation of $2P|\langle u_i, v_i \rangle-\operatorname{id}_{\langle u_i, v_i \rangle}$ is
\begin{equation*}
\begin{pmatrix}
\cos\theta_{i} & \sin\theta_{i}\\
\sin\theta_{i} & -\cos\theta_{i}
\end{pmatrix}.
\end{equation*}
\item[(ii)] The equality $Q u_1=u_1$ holds and $\langle v_i, u_{i+1} \rangle$ is invariant under $Q$ for any $i \in \mathbb{N}$.
The matrix representation of $2Q|\langle v_i, u_{i+1} \rangle-\operatorname{id}_{\langle v_i, u_{i+1} \rangle}$ is
\begin{equation*}
\begin{pmatrix}
\cos\omega_{i} & \sin\omega_{i}\\
\sin\omega_{i} & -\cos\omega_{i}
\end{pmatrix}.
\end{equation*}
\end{enumerate}
In particular, $V$ is separable.
The matrix representation of $2P-\operatorname{id}_{V}$ is
\begin{align*}
\left(\begin{array}{>{\centering\arraybackslash}p{0.3cm}>{\centering\arraybackslash}p{0.3cm}>{\centering\arraybackslash}p{0.3cm}>{\centering\arraybackslash}p{0.3cm}>{\centering\arraybackslash}p{0.3cm}>{\centering\arraybackslash}p{0.3cm}}
\cline{1-2}
\multicolumn{2}{|c|}{\multirow{2}{*}{$r(\theta_1)$}} & \multicolumn{1}{c}{} & \multicolumn{3}{c}{\multirow{3}{*}{\mbox{\Large $O$}}}\\
\multicolumn{1}{|c}{} & \multicolumn{1}{c|}{} & \multicolumn{1}{c}{} & & & \\ \cline{1-2}
\multicolumn{2}{c}{} & $\ddots$ & \multicolumn{3}{c}{}\\
 \cline{4-5}
\multicolumn{2}{c}{\multirow{3}{*}{\mbox{\Large $O$}}} & \multicolumn{1}{c}{\multirow{2}{*}{}} & \multicolumn{2}{|c|}{\multirow{2}{*}{$r(\theta_{n})$}} & \\
 & \multicolumn{1}{c}{} & \multicolumn{1}{c}{} & \multicolumn{1}{|c}{} & \multicolumn{1}{c|}{} &  \\ \cline{4-5}
 & \multicolumn{1}{c}{} & \multicolumn{1}{c}{} & \multicolumn{1}{c}{} & \multicolumn{1}{c}{} & $\ddots$ \\ 
\end{array}\right)
\end{align*}
and the matrix representation of $2Q-\operatorname{id}_V$ is
\begin{align*}
\left(
\begin{array}{>{\centering\arraybackslash}p{0.3cm}>{\centering\arraybackslash}p{0.3cm}>{\centering\arraybackslash}p{0.3cm}>{\centering\arraybackslash}p{0.3cm}>{\centering\arraybackslash}p{0.3cm}>{\centering\arraybackslash}p{0.3cm}>{\centering\arraybackslash}p{0.3cm}}
\cline{1-1}
\multicolumn{1}{|c|}{\multirow{1}{*}{$1$}} & $0$ & $0$ & \multicolumn{1}{c}{\multirow{3}{*}{}} & \multicolumn{3}{c}{\multirow{3}{*}{\mbox{\Large $O$}}}\\ \cline{1-3}
\multicolumn{1}{c}{0}  & \multicolumn{2}{|c|}{\multirow{2}{*}{$r(\omega_1)$}} & \multicolumn{1}{c}{} & & & \\
\multicolumn{1}{c}{0} & \multicolumn{1}{|c}{} & \multicolumn{1}{c|}{} & \multicolumn{1}{c}{} &  &  & \\  \cline{2-3}
\multicolumn{3}{c}{} & $\ddots$ & \multicolumn{3}{c}{} \\ \cline{5-6}
\multicolumn{3}{c}{\multirow{4}{*}{\mbox{\Large $O$}}} & \multicolumn{1}{c}{\multirow{3}{*}{}} & \multicolumn{2}{|c|}{\multirow{2}{*}{$r(\omega_{n-1})$}} & \multicolumn{1}{c}{} \\
 & & \multicolumn{1}{c}{} & \multicolumn{1}{c}{} & \multicolumn{2}{|c|}{} & \multicolumn{1}{c}{} \\ \cline{5-6}
 & & \multicolumn{1}{c}{} & \multicolumn{1}{c}{} & \multicolumn{1}{c}{} & \multicolumn{1}{c}{} & \multicolumn{1}{c}{\multirow{1}{*}{$\ddots$}} \\
\end{array}
\right),
\end{align*}
where $r$ is defined by $r(x)=\begin{pmatrix}\cos x & \sin x \\ \sin x & -\cos x \end{pmatrix}\quad(x \in \mathbb{R})$.
Moreover, we say that $(V,\{P,Q\})$ admits  a constant-angle one-shifted form
if and only if $\Theta,\Omega$ are constant and identical $(\text{so }\theta_1=\theta_i=\omega_i$ for any $i \in \mathbb{N})$.
\end{Definition}

We present two characterizations
of pairs of orthogonal projections that admit a one-shifted form.
The second characterization (ii) via cyclic vectors is employed in an extension of Jordan's lemma, which is discussed in Theorem~\ref{theorem:cyclicdecompositionunitaryrepresentationinfinitedihedralgroup}.
The third characterization (iii) via explicit approximations is employed to derive explicit formulas for finite dimensional approximations of pairs of orthogonal projections, which is stated in Theorem~\ref{theorem:twodimensionalapproximation}.

\begin{Proposition}
\label{prop:oneshiftedcyclicrepresentation}
Assume the setups $(S1), (S2)$ in {\rm{Definition~\ref{def:oneshiftedform}}},
and let
\begin{enumerate}
\item[(A1)] $\dim V=\infty$,
\end{enumerate}
then the following three conditions are equivalent.
\begin{enumerate}
\item[(i)] The pair $(V,\{P,Q\})$ admits a one-shifted form.
\item[(ii)] There exists a cyclic unit vector $u \in V$
for $\tau_{P,Q}$ such that $Qu=u$.
\item[(iii)] There exists an orthonormal basis $\mathcal{U}=\{u_1,v_1,u_2,v_2,\dots\}$ of $V$ and sequences $\Theta=(\theta_i)_{i=1}^\infty,\Omega=(\omega_i)_{i=1}^\infty$ of angles in $(0,\pi)$ such that \\$\{(\langle u_1,v_1,\dots,u_n,v_n \rangle,\{P_n(\Theta),Q_n(\Omega)\})\}_{n=1}^\infty$ is a finite dimensional approximation of $(V,\{P,Q\})$.
\end{enumerate}
\end{Proposition}
\begin{proof}[Proof of that (i) $\implies$ (ii) and (iii) $\implies$ (i):]
These implications are straightforward.
\end{proof}
\begin{proof}[Proof of that (i) $\implies$ (iii):]
Let $\mathcal{U}=\{u_1,v_1,u_2,v_2,\dots\}$ be an orthonormal basis of $V$
such that $(V,\{P,Q\})$ has a one-shifted form in $\mathcal{U}$.
Let $v \in V$ and $\varepsilon>0$.
For a sufficiently large $n \in \mathbb{N}$, there exists $w \in \langle u_1,v_1,\dots,u_n,v_n \rangle$
such that $\|v-w\|<\varepsilon$.
Then $Pw=P_{n+1}(\Theta)w$ and $Qw=Q_{n+1}(\Omega)w$.
Hence, condition (iii) is satisfied.
\end{proof}
\begin{proof}[Step~1 in the proof of (ii) $\implies$ (i):]
In this step, we construct $u_1, v_1, u_2$.
Define $D$ and $u_1$ by $D=\{v \in V|\ \|v\|=1\}, u_1=u$.
If $Pu_1 \in \langle u_1 \rangle$,
$\langle u_1 \rangle$ is an invariant subspace under $\tau_{P,Q}$.
Since $u$ is a cyclic vector, we have $\langle u_1 \rangle=V$. That contradicts assumption (A1).
Therefore, the vector $Pu_1$ is not in $\langle u_1 \rangle$.
Hence, there exists $v_1 \in D$ such that $v_1 \in \langle u_1, P u_1 \rangle$
and $v_1 \perp u_1$
and $\langle u_1, v_1 \rangle$ is invariant under $P$.
For any $\phi_1 \in \mathbb{R}$, 
$\langle u_1, \exp(i\phi_1)v_1 \rangle$ is also invariant under $P$ and $P$ is positive semidefinite on $\langle u_1, \exp(i\phi_1)v_1 \rangle$.
Therefore, we may assume that the matrix representation of $P$ on $\langle u_1, \exp(i\phi_1)v_1 \rangle$
is
$
\begin{pmatrix}
x & y\\
y & z
\end{pmatrix}
$
for some $x,y,z \in [0,\infty)$.
Since
$
\begin{pmatrix}
x & y\\
y & z
\end{pmatrix}^2
=
\begin{pmatrix}
x & y\\
y & z
\end{pmatrix},
$
we have
$
x^2+y^2=x, xy+yz=y
$.
Therefore, the equalities $y=\sqrt{x(1-x)}, z=1-x$ hold.
It follows that there exists $\displaystyle \psi_1 \in \left(0,\dfrac{\pi}{2}\right)$ such that
$\sqrt{x}=\cos \psi_1,\sqrt{1-x}=\sin \psi_1$.
Therefore, we have $
\begin{pmatrix}
x & y\\
y & z
\end{pmatrix}
=
\begin{pmatrix}
\cos^2 \psi_1 & \cos \psi_1\sin \psi_1\\
\cos \psi_1 \sin \psi_1 & 1-\cos^2 \psi_1
\end{pmatrix}
$.
Moreover, it follows that
\begin{align*}
\begin{pmatrix}
\cos^2 \psi_1 & \cos \psi_1\sin \psi_1\\
\cos\psi_1\sin\psi_1 & 1-\cos^2 \psi_1
\end{pmatrix}
&=
\dfrac{1}{2}
\begin{pmatrix}
\cos(2\psi_1)+1 & \sin(2\psi_1)\\
\sin(2\psi_1) & 1-\cos(2\psi_1)
\end{pmatrix}\\
&=
\dfrac{1}{2}
\begin{pmatrix}
\cos\theta_1 & \sin\theta_1\\
\sin\theta_1 & -\cos\theta_1
\end{pmatrix}
+\dfrac{1}{2}E_2,
\end{align*}
where $\theta_1=2\psi_1 \in (0,\pi)$.

If $Qv_1 \in \langle u_1, v_1 \rangle$,
$\langle u_1, v_1 \rangle$ is an invariant subspace under $\tau_{P,Q}$.
Since $u$ is a cyclic vector, we have $\langle u_1, v_1 \rangle=V$.
It follows that $\dim V \le 2$. That contradicts assumption (A1).
Therefore, we have $Q v_1 \not\in \langle u_1, v_1 \rangle$.
It follows that there exists $u_2 \in D$ such that $u_2 \in \langle v_1, Q v_1 \rangle$ and $u_2 \perp v_1$.
Then $\langle v_1, u_2 \rangle$ is invariant under $Q$.
Since $u_2=a v_1 + b Qv_1$ for some $a, b \in \mathbb{C}$,
we have $(u_2,u_1)=a(v_1,u_1)+b(Qv_1,u_1)=b(v_1,Qu_1)=0$.
Therefore, we conclude that $u_2 \perp \langle u_1, v_1 \rangle$.
Following the same reasoning as above, we may assume that
there exists $\omega_1 \in (0,\pi)$ such that
the matrix representation of $Q|\langle v_1, u_2 \rangle$ is
\begin{align*}
\dfrac{1}{2}
\begin{pmatrix}
\cos\omega_1 & \sin\omega_1\\
\sin\omega_1 & -\cos\omega_1
\end{pmatrix}
+\dfrac{1}{2}E_2.
\end{align*}

Consequently, 
the vectors $u_1,v_1$ and $u_2$ lie in $D$, are mutually orthogonal, and satisfy the following conditions:
\begin{itemize}
\item The spanned subspace $\langle u_1, v_1 \rangle$ is invariant under $P$, and the matrix representation of $P|\langle u_1, v_1 \rangle$
is
\begin{align*}
\dfrac{1}{2}
\begin{pmatrix}
\cos \theta_1 & \sin \theta_1\\
\sin \theta_1 & -\cos \theta_1
\end{pmatrix}
+\dfrac{1}{2}E_2\quad (\theta_1 \in (0,\pi)).
\end{align*}
\item The equality $Q u_1 = u_1$ holds, and $\langle v_1, u_2 \rangle$ is invariant under $Q$, and the matrix representation of $Q|\langle v_1, u_2 \rangle$
is
\begin{align*}
\dfrac{1}{2}
\begin{pmatrix}
\cos \omega_1 & \sin \omega_1\\
\sin \omega_1 & -\cos \omega_1
\end{pmatrix}
+\dfrac{1}{2}E_2\quad(\omega_1 \in (0,\pi)).
\end{align*}
\end{itemize}
\end{proof}
\begin{proof}[Step~2 in the proof of (ii) $\implies$ (i):]
In this step, we construct the orthogonal basis $\{u_1, v_1,u_2,v_2,\dots\}$ by induction.
Let $i \in \mathbb{N}$ and assume that we get $\{u_1,\dots,u_i,v_i, u_{i+1}\}$
satisfy the claim in condition (i).
Following the same reasoning as Step~1,
we have $Pu_{i+1} \not\in \langle u_1,\dots,u_i,v_i,u_{i+1} \rangle$.
Hence, there exists $v_{i+1} \in D$ such that $v_{i+1} \in \langle u_{i+1},P u_{i+1} \rangle$ and $v_{i+1} \perp u_{i+1}$
and $\langle u_{i+1}, v_{i+1} \rangle$ is invariant under $P$.
As in Step~1, 
we may assume that 
the matrix representation of $P|\langle u_{i+1}, v_{i+1} \rangle$ is
\begin{align*}
\dfrac{1}{2}
\begin{pmatrix}
\cos \theta_{i+1} & \sin \theta_{i+1}\\
\sin \theta_{i+1} & -\cos \theta_{i+1}
\end{pmatrix}
+\dfrac{1}{2}E_2\quad\text{ for some $\theta_{i+1} \in (0,\pi)$}.
\end{align*}

Let $w \in \langle u_1,\dots,u_i,v_i \rangle$.
Since $v_{i+1}=a u_{i+1} + b Pu_{i+1}$ for some $a, b \in \mathbb{C}$ and $\langle u_1,\dots,u_i,v_i \rangle$ is invariant under $P$,
we have $(v_{i+1},w)=a(u_{i+1},w)+b(Pu_{i+1},w)=b(u_{i+1},Pw)=0$.
Therefore, we have
\begin{equation*}
v_{i+1} \perp \langle u_1,v_1,\dots,u_i,v_i,u_{i+1} \rangle.
\end{equation*}

As in Step~1, it follows that $Qv_{i+1} \not\in \langle u_1,v_1,\dots,u_{i+1},v_{i+1} \rangle$.
Hence, there exists $u_{i+2} \in D$ such that $u_{i+2} \in \langle v_{i+1},Q v_{i+1} \rangle$ and $u_{i+2} \perp v_{i+1}$.
Let $w \in \langle u_1,v_1,\dots,u_{i+1} \rangle$.
Since $u_{i+2}=a v_{i+1} + b Qv_{i+1}$ for some $a, b \in \mathbb{C}$ and $\langle u_1,v_1,\dots,u_{i+1} \rangle$ is invariant under $Q$,
we have $(u_{i+2},w)=a(v_{i+1},w)+b(Qv_{i+1},w)=b(v_{i+1},Qw)=0$.
Therefore, we have
\begin{equation*}
u_{i+2} \perp \langle u_1,v_1,\dots,u_{i+1},v_{i+1} \rangle
\end{equation*}
and $\langle v_{i+1}, u_{i+2} \rangle$ is invariant under $Q$.

Following the same reasoning as Step~1, 
we may assume that there exists $\omega_{i+1} \in (0,\pi)$ such that
the matrix representation of $Q|\langle v_{i+1}, u_{i+2} \rangle$ is
\begin{align*}
\dfrac{1}{2}
\begin{pmatrix}
\cos \omega_{i+1} & \sin \omega_{i+1}\\
\sin \omega_{i+1} & -\cos \omega_{i+1}
\end{pmatrix}
+\dfrac{1}{2}E_2.
\end{align*}
Thus we get $u_1,\dots,u_{i+1},v_{i+1},u_{i+2}$ satisfying the claim in condition (i).

By induction,
we get $\mathcal{U}=\{u_1,v_1,\dots\} \subset D$
such that 
\begin{itemize}
\item For any $i \in \mathbb{N}$, the spanned subspace $\langle u_i, v_i \rangle$ is invariant under $P$ and
the matrix representation of $P|\langle u_{i}, v_{i} \rangle$ is
\begin{align*}
\dfrac{1}{2}
\begin{pmatrix}
\cos \theta_{i} & \sin\theta_{i}\\
\sin \theta_{i} & -\cos\theta_{i}
\end{pmatrix}
+\dfrac{1}{2}E_2\quad(\theta_i \in (0,\pi)).
\end{align*}
\item The equality $Q u_1=u_1$ holds, and for any $i \in \mathbb{N}$ $\langle v_i, u_{i+1} \rangle$ is invariant under $Q$ and
the matrix representation of $Q|\langle v_{i}, u_{i+1} \rangle$ is
\begin{align*}
\dfrac{1}{2}
\begin{pmatrix}
\cos \omega_{i} & \sin \omega_{i}\\
\sin \omega_{i} & -\cos \omega_{i}
\end{pmatrix}
+\dfrac{1}{2}E_2\quad(\omega_i \in (0,\pi)).
\end{align*}
\end{itemize}
Moreover, $\langle u_1,v_1,\dots,u_i,v_i,\dots \rangle$ is invariant under $\tau_{P,Q}$.
From the assumption that $u$ is a cyclic vector,
the set $\mathcal{U}$ is an orthonormal basis of $V$.
\end{proof}

We introduce the following notation for constant-angle one-shifted forms.

\begin{Notation}
Let $\theta \in (0,\pi)$ and define $\Theta=\{\theta_i\}_{i=1}^\infty$
and $\Omega=\{\omega_i\}_{i=1}^\infty$ by $\theta_i=\omega_i=\theta$ $(i=1,2,\dots)$.
Define self-adjoint operators $A(\theta), B(\theta)$ on $\ell^2$ by
$A(\theta)=A(\Theta)$ and $B(\theta)=B(\Omega)$.
We identify $\mathbb{C}^n$ with
\\$\left\{\{a_i\}_{i=1}^n \in \ell^2 \middle|a_{n+k} = 0\quad (k=1,2,\dots)\right\}$.
Define the finite dimensional approximation $\mathcal{A}$ of $(\ell^2,\{P(\theta),Q(\theta)\})$ by
$\mathcal{A}=\{(\mathbb{C}^n, \{P_n(\theta), Q_n(\theta)\})\}_{n=1}^\infty$,
where $P_n(\theta)=P_n(\Theta)$, $Q_n(\theta)=Q_n(\Omega)$, $A_n(\theta)=2P_n(\theta)-E_{2n}$, $B_n(\theta)=2Q_n(\theta)-E_{2n}$ $(n=1,2,\dots)$.
\end{Notation}

\subsection{Extension of Jordan's lemma}
\label{subsec:jordan}

An extension of Jordan's lemma is stated in Theorem~\ref{theorem:cyclicdecompositionunitaryrepresentationinfinitedihedralgroup}.
From the viewpoint of group representation theory, Theorem~\ref{theorem:cyclicdecompositionunitaryrepresentationinfinitedihedralgroup} provides an explicit formula that decomposes every infinite dimensional unitary representation of $D_\infty$ into cyclic representations.

\begin{Theorem}
\label{theorem:cyclicdecompositionunitaryrepresentationinfinitedihedralgroup}
Assume the setups $(S1), (S2)$ in {\rm{Definition~\ref{def:oneshiftedform}}},
then there exist at most countable closed invariant subspaces $\{W_i\}_{i \in I}$ under $\tau_{P,Q}$ such that
\begin{enumerate}
\item[(i)] The orthogonal direct sum decomposition $\displaystyle V=\bigoplus_{i \in I}W_i$ holds.
\item[(ii)] For any $i \in I$, the subspace $W_i$ is at most two dimensional or infinite dimensional.
\item[(iii)] For any $i \in I$ such that $\dim W_i \le 2$, $W_i$ is irreducible under $\tau_{P,Q}$.
\item[(iv)] For any $i \in I$ such that $\dim W_{i}=\infty$, the pair $(W_i,\{P|W_i,Q|W_i\})$ has a one-shifted form
and $W_i$ does not admit nonzero finite dimensional invariant subspaces under $\tau_{P,Q}$.
\end{enumerate}
\end{Theorem}
\begin{proof}
We define the sets $\mathcal{B}, \mathcal{C}, \mathcal{D}, \mathcal{T}$ by
\begin{align*}
\mathcal{B}&=\{W|W\text{ is an invariant subspace of $V$ under $\tau_{P, Q}$ and $\dim W \le 2$}\},\\
\mathcal{C}&=\{W|W\text{ is an infinite dimensional invariant subspace of $V$ under $\tau_{P, Q}$} \\
&\text{   and admits no nonzero finite dimensional invariant subspaces under $\tau_{P,Q}$}\\
&\text{   and there exists a cyclic vector $u \in W$ such that $Qu=u$}\},\\
\mathcal{D}&=\mathcal{B} \cup \mathcal{C},\\
\mathcal{T}&=\left\{D \subset \mathcal{D}\middle|W_1 \perp W_2\quad(\text{for any $W_1, W_2 \in D$ such that }  W_1 \ne W_2)\right\}.
\end{align*}

It is straightforward that $\mathcal{T}$ is not empty.
We assume that the order of $\mathcal{T}$ is defined by set inclusion.
Let $T$ be a totally ordered subset of $\mathcal{T}$.
It follows directly that $\displaystyle \bigcup_{D \in T}D \in \mathcal{T}$.
From Zorn's Lemma, the set $\mathcal{T}$ has a maximal
element $D$.
We define $V_0$ by
$\displaystyle
V_0=\bigoplus_{W \in D}W.
$
To derive a contradiction,
we assume $V_0 \ne V$.
If $Q|V_0^{\perp}=0$ or $V_0^{\perp}$ admits a nonzero finite dimensional invariant subspace of $V$ under $\tau_{P,Q}$,
from Proposition~\ref{fact:infinitedihedralgroupfinitedimensionalrepresentationdecomposition},
the subspace $V_0^{\perp}$ has an at most two dimensional
invariant subspace $W$ under $\tau_{P,Q}$.
Then $W \in \mathcal{B}$,
and hence $D \cup \{W\} \in \mathcal{T}$.
That contradicts the fact that $D$ is a maximal element of $\mathcal{T}$.

Assume $Q|V_0^{\perp}\ne 0$ and $V_0^{\perp}$ admits no nonzero finite dimensional invariant subspaces of $V$ under $\tau_{P,Q}$.
Then there exists a unit vector $u \in V_0^{\perp}$ such that $Qu = u$.
Define $W$ by $W=\overline{\langle \tau_{P,Q}(D_\infty)u \rangle}$.
It is straightforward that $W \perp V_0$.
Since $V_0^{\perp}$ does not admit nonzero finite dimensional invariant subspaces of $V$ under $\tau_{P,Q}$,
the subspace $W$ also does not admit such subspaces and $\dim W=\infty$.
Then $W \in \mathcal{C}$ and $W \perp W'$ for any $W' \in D$.
Therefore, we have $D \cup \{W\} \in \mathcal{T}$.
That contradicts the fact that $D$ is a maximal element of $\mathcal{T}$.
Therefore, we have $\displaystyle V=\bigoplus_{W \in D}W$.
Since $V$ is separable,
$D$ must be at most countable.
From Proposition~\ref{prop:oneshiftedcyclicrepresentation},
for any $W \in \mathcal{C}$, $(W,\{P|W,Q|W\})$ has a one-shifted form.
\end{proof}

Note that 
there exists a pair of orthogonal projections $P$ and $Q$ on an infinite dimensional Hilbert space $V$ satisfying the following conditions:
\begin{itemize}
\item The pair $(V,\{P,Q\})$ admits a constant-angle one-shifted form.
\item There does not exist a nonzero finite dimensional invariant subspace under $P$ and $Q$.
\end{itemize}

\begin{Example}
\label{example:properinfinitedimensionalcyclicrepresentation}
Let
\begin{enumerate}
\item[(S1)] $V,V_1,V_2$ be defined by
\begin{align*}
V&=\ell^2=\left\{\{x_j\}_{j=1}^\infty\subset \mathbb{C}\middle|\sum_{j=1}^\infty |x_j|^2<\infty\right\},\\
V_1&=\left\{\{x_j\}_{j=1}^\infty \in \ell^2\middle|x_{2k+1}=x_{2k+2}\quad(k=0,1,2,\dots)\right\},\\
V_2&=\left\{\{x_j\}_{j=1}^\infty \in \ell^2\middle|x_{2k+2}=x_{2k+3}\quad(k=0,1,2,\dots)\right\},
\end{align*}
\item[(S2)] $P_i$ for $i \in \{1,2\}$ be the orthogonal projection onto $V_i$ 
and $A_i$ for $i \in \{1,2\}$ be defined by $A_i=2P_i-\operatorname{id}_V$,
\end{enumerate}
then
\begin{enumerate}
\item[(i)] Let $x=\{x_j\}_{j=1}^\infty \in V$. Then the following holds:
\begin{align*}
P_1x&=(\dfrac{1}{2}(x_1+x_2),\dfrac{1}{2}(x_1+x_2),\dots,\\
&\quad\dfrac{1}{2}(x_{2k+1}+x_{2k+2}),\dfrac{1}{2}(x_{2k+1}+x_{2k+2}),\dots),\\
A_1x&=(x_2,x_1,\dots,
x_{2k+2},x_{2k+1},\dots),\\
P_2x&=(x_1,\dfrac{1}{2}(x_2+x_3),\dfrac{1}{2}(x_2+x_3),\dots,\\
&\quad\dfrac{1}{2}(x_{2k+2}+x_{2k+3}),\dfrac{1}{2}(x_{2k+2}+x_{2k+3}),\dots),\\
A_2x&=(x_1,x_3,x_2,\dots,
x_{2k+3},x_{2k+2},\dots).
\end{align*}
\item[(ii)] Let $x=\{x_j\}_{j=1}^\infty \in V$ and $k \in \mathbb{N}$. Then
\begin{align*}
(A_1A_2x)_{4k+1}&=x_{4k+3},&(A_1A_2x)_{4k+3}&=x_{4k+5},\\
(A_1A_2x)_{4k+2}&=x_{4k},&(A_1A_2x)_{4k+4}&=x_{4k+2},\\
(A_2A_1x)_{4k+1}&=x_{4k-1},&(A_2A_1x)_{4k+3}&=x_{4k+1},\\
(A_2A_1x)_{4k+2}&=x_{4k+4},&(A_2A_1x)_{4k+4}&=x_{4k+6}.
\end{align*}
\item[(iii)] The pair $(V,\{P_1,P_2\})$ admits a constant-angle one-shifted form whose angle is $\dfrac{\pi}{2}$ in the standard basis.
\item[(iv)] There does not exist a nonzero finite dimensional invariant subspace under $P_1$ and $P_2$.
\end{enumerate}
\end{Example}
\begin{proof}[Proof of (i)-(iii)]
These follow directly from the definitions of $P_1$, $P_2$, $A_1$, and $A_2$, and straightforward calculations.
\end{proof}\begin{proof}[Proof of (iv)]
Let $W$ be a nonzero invariant subspace under $\tau_{P_1, P_2}$.
From Proposition~\ref{fact:infinitedihedralgroupfinitedimensionalrepresentationdecomposition},
it is enough to show that $\dim W \ge 3$.
Since $W$ is nonzero, there exists $w \in W \setminus \{0\}$.
From statement (ii), there are $m \in \mathbb{N}$ and a sequence $i_1,\dots,i_m \in \{1,2\}$
and $\alpha \in \mathbb{C}$ such that
the first component of $\alpha A_{i_1}\dots A_{i_m}w$ is $1$.
Without loss of generality,
we may assume that $w_1=1$.
Since $w \in \ell^2$, we have $\displaystyle \lim_{n \to \infty}w_n=0$.
Therefore, there exists $i_0 \in [2,\infty) \cap \mathbb{N}$
such that $|w_{i_0}| \ne 1$ and $w_j=1$ for any $j \in \{1,\dots,i_0-1\}$.
From statement (ii), there exists a finite sequence $j_1,\dots,j_n \in \{1,2\}$ such that
$
A_{j_1}\dots A_{j_n} w = (1,w'_2,\dots),
A_1 A_{j_1}\dots A_{j_n} w = (w'_2, 1,\dots),
$
where $|w'_2| \ne 1$.
Since $(w'_2)^2 \ne 1$, these two vectors are linearly independent.
Hence, there exist $u, v \in W$ such that
$
u=(1,0,u_3,\dots), v=(0, 1, v_3, \dots)
$.

To derive a contradiction, assume that $\dim W = 2$.
Since $A_2 u = (1,u_3,\dots)$ is in $\langle u, v \rangle$, we have $u_3=0$.
Since $A_2 v = (0,v_3,\dots)$ is in $\langle u, v \rangle$, we have $|v_3|=1$.
Therefore, it follows that
$u=(1,0,0, u_4,\dots), v=(0, 1, v_3, v_4, \dots)$ and $|v_3|=1$.
Since $A_1 u = (0, 1,u_4, 0,\dots)$ is in $\langle u, v \rangle$, we have $|u_4|=1$.
Since $A_1 v = (1, 0,v_4, v_3,\dots)$ is in $\langle u, v \rangle$, we have $v_4=0$.
Therefore, the following holds for $k=0$:
\begin{align*}
|u_{4k+1}|&=1, &u_{4k+2}&=0, &u_{4k + 3}&=0, &|u_{4k + 4}|&=1,\\
v_{4k+1}&=0, &|v_{4k+2}|&=1, &|v_{4k + 3}|&=1, &v_{4k + 4}&=0.
\end{align*}

Let $k_0 \in [0, \infty) \cap \mathbb{Z}$
and assume that the above statement holds for $k=k_0$.
From statement (ii), the following equalities hold for some $\{a_i\}_{i=1}^8 \subset \mathbb{C}$:
\begin{align*}
(A_1A_2)^{k_0 + 1}u&=(0, a_1, u_{4(k_0+1)+1}, \dots),\\
(A_1A_2)^{k_0 + 1}v&=(v_{4k_0+3}, a_2, v_{4(k_0+1)+1}, \dots),\\
(A_1A_2)^{k_0 + 1}A_1u&=(u_{4(k_0+1)}, a_3, u_{4(k_0+1)+2}, \dots),\\
(A_1A_2)^{k_0 + 1}A_1v&=(0, a_4, v_{4(k_0+1)+2}, \dots),\\
(A_1A_2)^{k_0 + 2}u&=(u_{4(k_0+1)+1}, a_5, u_{4(k_0+1)+3}, \dots),\\
(A_1A_2)^{k_0 + 2}v&=(v_{4(k_0+1)+1}, a_6, v_{4(k_0+1)+3}, \dots),\\
(A_1A_2)^{k_0 + 2}A_1u&=(u_{4(k_0+1)+2}, a_7, u_{4(k_0+1)+4}, \dots),\\
(A_1A_2)^{k_0 + 2}A_1v&=(v_{4(k_0+1)+2}, a_8, v_{4(k_0+1)+4}, \dots),
\end{align*}
and these vectors are in $\langle u, v \rangle$.
Hence, the following equalities hold:
\begin{align*}
|u_{4(k_0+1)+1}| &= 1, &v_{4(k_0+1)+1} &= 0\\
u_{4(k_0+1)+2} &= 0, &|v_{4(k_0+1)+2}| &= 1\\
u_{4(k_0+1)+3} &= 0, &|v_{4(k_0+1)+3}| &= 1\\
|u_{4(k_0+1)+4}| &= 1, &v_{4(k_0+1)+4} &= 0.
\end{align*}

By induction, the above statement holds for any $k \in [0, \infty) \cap \mathbb{Z}$.
This conclusion contradicts the fact that $\displaystyle \lim_{k \to \infty}u_k=0$.
\end{proof}

\subsection{Explicit construction of finite dimensional approximations}
\label{subsec:explicit}

From Theorem~\ref{theorem:cyclicdecompositionunitaryrepresentationinfinitedihedralgroup}
and Proposition~\ref{proposition:finite_dimensional_approximation_directsum},
we obtain an explicit formula for finite dimensional approximation
of pairs of orthogonal projections, as follows.

\begin{Theorem}
\label{theorem:twodimensionalapproximation}
Assume the setting of {\rm{Theorem~\ref{theorem:cyclicdecompositionunitaryrepresentationinfinitedihedralgroup}}},
and let
\begin{enumerate}
\item[(S3)] $\{I_n\}_{n=1}^\infty$ be a sequence of finite subsets  of $I$ such that $I_n \subset I_{n+1}$  for any $n \in \mathbb{N}$ and $\displaystyle\bigcup_{n=1}^\infty I_n=I$,
\item[(S4)] $\mathcal{U}=\{u_{i,1},v_{i,1},u_{i,2},v_{i,2},\dots\}$ be an orthonormal basis of $W_i$, and $\Theta_i=\{\theta_{i,j}\}_{j=1}^\infty$ and $\Omega_i=\{\omega_{i,j}\}_{j=1}^\infty$ be
in $(0,\pi)$ for $i \in I$ such that $\dim W_i=\infty$,
\item[(S5)] $(W_i,P|W_i,Q|W_i)$ admits a one-shifted form in $\mathcal{U}$
whose angles are $\Theta_i, \Omega_i$ for any $i \in I$ such that $\dim W_i=\infty$, 
\item[(S6)] for $n \in \mathbb{N}$ and $i \in I_n$, $W_{i,n}$ be defined by
\begin{equation*}
W_{i,n}=
\left\{
\begin{array}{ll}
\langle u_{i,1},v_{i,1},\dots,u_{i,n},v_{i,n} \rangle & \quad \dim W_i=\infty, \\
W_i & \quad \text{otherwise}, 
\end{array}
\right.
\end{equation*}
\item[(S7)] for $n \in \mathbb{N}$ and $i \in I_n$, $P_{i,n}$ and $Q_{i,n}$ be defined by
\begin{equation*}
P_{i,n}=
\left\{
\begin{array}{ll}
P_{n}(\Theta_i) & \quad \dim W_i=\infty, \\
P|W_i & \quad \text{otherwise},
\end{array}
\right. \quad
Q_{i,n}=
\left\{
\begin{array}{ll}
Q_{n}(\Omega_i) & \quad \dim W_i=\infty, \\
Q|W_i & \quad \text{otherwise},
\end{array}
\right.
\end{equation*}
and
\item[(S8)] $V_n, P_n,Q_n$ for $n \in \mathbb{N}$ be defined by
\begin{align*}
V_n=\bigoplus_{i \in I_n}W_{i,n},
P_n=\bigoplus_{i \in I_n}P_{i,n},
Q_n=\bigoplus_{i \in I_n}Q_{i,n},
\end{align*}
\end{enumerate}
then the sequence $\{(V_n,\{P_n,Q_n\})\}_{n=1}^\infty$ is a finite dimensional approximation of $(V,\{P,Q\})$.
\end{Theorem}

\section{Spectral radius of commutators of two orthogonal projections}
\label{sec:chsh}

We present a new method for estimating
 the spectral radius $\rho([P,Q])=\rho([A,B])$
for a pair of orthogonal projections $P$ and $Q$ on a Hilbert space $V$,
where $A=2P-\operatorname{id}_V$ and $B=2Q-\operatorname{id}_V$.
As stated in Section~\ref{subsec:chshoperatorintroduction},
this problem is equivalent to estimating the spectral radius of the Bell-CHSH operator.
Moreover, from Theorem~\ref{theorem:cyclicdecompositionunitaryrepresentationinfinitedihedralgroup}, we can reduce the problem of estimating $\rho([A, B])$ to the following two cases:
\begin{quote}
\begin{enumerate}
\item[Case~1] $\tau_{P,Q}$ admits an orthogonal direct sum decomposition into finite dimensional invariant subspaces under $\tau_{P,Q}$,
\item[Case~2] $\tau_{P,Q}$ does not satisfy Case~1 and admits a one-shifted form.
\end{enumerate}
\end{quote}
Case~1 is trivial. Hence we mainly consider Case~2.

In Section~\ref{subsec:inequaity}, we give
upper and lower estimates of $\rho([A,B])$
in Case~1 and Case~2.

In Section~\ref{subsec:equality}, we prove that
these estimates become exact when $P$ and $Q$ admit a constant-angle one-shifted form.

In Section~\ref{subsec:chsh}, we give
upper and lower estimates of the spectral radius of the Bell-CHSH operator in 
the case where every subsystem of the bipartite quantum system satisfies Case~1 or Case~2.

\subsection{Upper and lower estimates of the spectral radius}
\label{subsec:inequaity}

In this section, we use the following notation.

\begin{Notation}
\label{notation:finitedimensionalspectrumsum}
Let
\begin{enumerate}
\item[(S1)] $V$ be a separable Hilbert space,
\item[(S2)] $P$ and $Q$ be orthogonal projections on $V$,
\item[(S3)] $A$ and $B$ be defined by $A=2P-\operatorname{id}_V, B=2Q-\operatorname{id}_V$, and
\item[(S4)] $\mathcal{A}=\{(V_{n},\{P_{n},Q_{n}\})\}_{n=1}^\infty$ be a finite dimensional approximation of \\$(V,\{P,Q\})$.
\end{enumerate}
With this notation,
let $F(A+B, \mathcal{A})$
be the set of all real numbers $\lambda \in (0,2)$
such that $-\lambda$ and $\lambda$ are in
the finite dimensional approximate point spectrum
of $A+B=2(P+Q-\operatorname{id}_V)$ corresponding to $\mathcal{A}$.
Moreover, let $\tilde{F}(A+B, \mathcal{A})$ be defined by
\begin{equation*}
\tilde{F}(A+B, \mathcal{A})=F(A+B, \mathcal{A}) \cup \{0\}.
\end{equation*}
We denote $F(A+B, \mathcal{A}), \tilde{F}(A+B, \mathcal{A})$
by $F(A+B), \tilde{F}(A+B)$, respectively, whenever no confusion arises.
\end{Notation}

\begin{Notation}
\label{notation:finitedimensionalspectrumsumtrivial}
Assume the setups (S1)-(S3) in {\rm{Notation~\ref{notation:finitedimensionalspectrumsum}}},
and that
\begin{enumerate}
\item[(A1)] there exists an at most countable family $\{W_i\}_{i \in I}$
of invariant subspaces under $P$ and $Q$ such that
\begin{equation*}
V=\bigoplus_{\substack{i \in I}} W_i, W_{i_1} \perp W_{i_2}\ (\text{for any $i_1 \ne i_2 \in I$}), \text{ and  let }
\end{equation*}
\item[(S4)] $\{I_n\}_{n=1}^\infty$ be a family of finite subsets
of $I$ such that $I_n \subset I_{n+1}$ for any $n \in \mathbb{N}$ and $\displaystyle \bigcup_{n=1}^\infty I_n=I$.
\end{enumerate}
Under these assumption we define $\mathcal{A}$ by
\begin{align*}
\mathcal{A}&=\{(V_{n},\{P_{n},Q_{n}\})\}_{n=1}^\infty,\\
V_n&=\bigoplus_{i \in I_n} W_i, P_n=\bigoplus_{i \in I_n} P| W_i,
Q_n=\bigoplus_{i \in I_n} Q| W_i\quad (n=1,2,\dots)
\end{align*}
Remak that $\mathcal{A}$ is
a finite dimensional approximation of $(V,\{P,Q\})$.
\end{Notation}

In Theorem~\ref{theorem:refinementtsirelsoninequality}, which is the main result of this section, we provide upper and lower bounds for $\rho([A,B])$.
The upper bound is established in the general setting, without assuming Case~1 or Case~2. 
The lower bound is derived under the assumption that either Case~1 or Case~2 holds. 

\begin{Theorem}
\label{theorem:refinementtsirelsoninequality}
Assume the setups in {\rm{Notation~\ref{notation:finitedimensionalspectrumsum}}},
and let
\begin{enumerate}
\item[(S5)] $A_n$ and $B_n$ be defined by
$A_n=2P_n-\operatorname{id}_{V_n}$, $B_n=2Q_n-\operatorname{id}_{V_n}$ for $n \in \mathbb{N}$,
\end{enumerate}
then the following statements hold:
\begin{enumerate}
\item[(i)] Let $\lambda_{n}$ be an eigenvalue of $A_{n}+B_{n}$
such that $\lambda_{n}^2$ is closest to $2$ for $n \in \mathbb{N}$. Then
\begin{align*}
\rho([A,B]) \le \liminf_{n \to \infty}b(\lambda_{n}).
\end{align*}
\item[(ii)] Assume that $(V,\{P,Q\})$ has a one-shifted form and
$\mathcal{A}$ ( in (S4) of {\rm{Notation~\ref{notation:finitedimensionalspectrumsum}}} ) is
the finite dimensional approximation of $(V,\{P,Q\})$ defined in {\rm{Theorem~\ref{theorem:twodimensionalapproximation}}}.
Then
\begin{align*}
\sup_{\lambda \in \tilde{F}(A+B, \mathcal{A})}b(\lambda) \le \rho([A,B]).
\end{align*}
\item[(iii)] Assume that $(V,\{P,Q\})$ satisfies the conditions (S1)-(S3), (A1)
in {\rm{Notation~\ref{notation:finitedimensionalspectrumsumtrivial}}} and
$\mathcal{A}$ is
the finite dimensional approximation of $(V,\{P,Q\})$ defined in {\rm{Notation~\ref{notation:finitedimensionalspectrumsumtrivial}}}.
Then
\begin{equation*}
\sup_{\lambda \in \tilde{F}(A+B, \mathcal{A})}b(\lambda)=\rho([A,B]) = \liminf_{n \to \infty}b(\lambda_{n}).
\end{equation*}
\end{enumerate}
\end{Theorem}

\ifthenelse{\value{arxiv} = 1}{
After the submission of the first version of this preprint \cite{YFujii2025}, B\"ottcher and Spitkovsky \cite{ABottcher2026} obtained an infinite dimensional extension of Proposition~\ref{prop:commutorspectrum_finitedimension}.
}{
After the submission of the preprint \cite{YFujii2025} of this paper and after the submission of this paper itself, B\"ottcher and Spitkovsky \cite{ABottcher2026} obtained an infinite dimensional extension of Proposition~\ref{prop:commutorspectrum_finitedimension}, from which the part $F(A+B,\mathcal{A}) \subset \sigma(A+B)$ of our Theorem~\ref{theorem:refinementtsirelsoninequality} (ii) can be alternatively derived.
}
Hereafter, we provide a proof of Theorem~\ref{theorem:refinementtsirelsoninequality}
without their extension.

As a preparation for the proof of Theorem~\ref{theorem:refinementtsirelsoninequality}, we first identify 
$F(A+B)$ in Case~1. In Case~1, Proposition~\ref{fact:infinitedihedralgroupfinitedimensionalrepresentation}
allows us to directly identify set $F(A+B)$, as follows.

\begin{Lemma}
\label{lemma:finitedimensionalspectrumsumtrivial}
Assume the setups in {\rm{Notation~\ref{notation:finitedimensionalspectrumsumtrivial}}}.
Then we have
\begin{equation*}
F(A+B,\mathcal{A})=2\overline{\bigcup_{i \in I} \sigma((P+Q)|W_i-\operatorname{id}_{W_i})} \cap (0,2).
\end{equation*}
\end{Lemma}

We give the proof of Theorem~\ref{theorem:refinementtsirelsoninequality} as follows.

\begin{proof}[Proof of (i) of {\rm{Theorem~\ref{theorem:refinementtsirelsoninequality}}}:]
Let $\{(V_{\phi(j)},\{P_{\phi(j)},Q_{\phi(j)}\})\}_{j=1}^\infty$
be a subsequence of $\mathcal{A}$
such that $\{b(\lambda_{\phi(j)})\}_{j=1}^\infty$ converges,
and let $v \in V$ and $\varepsilon > 0$.
From Proposition~\ref{proposition:finite_dimensional_approximation_polynomial},
there exists $n_0 \in \mathbb{N}$ such that
 for any $n \in [n_0,\infty) \cap \mathbb{N}$, we have
$|([A,B]v,v)|-\varepsilon<|([A_n,B_n]\pi_nv,\pi_nv)|$.
From Proposition~\ref{fact:infinitedihedralgroupfinitedimensionalrepresentation},
for any $n \in [n_0,\infty) \cap \mathbb{N}$
\begin{equation*}
|([A_n,B_n]\pi_nv,\pi_nv)|
\le \rho([A_n,B_n]) = b(\lambda_n).
\end{equation*}
Hence, we have
$
\displaystyle
|([A,B]v,v)|-\varepsilon \le \liminf_{n \to \infty}b(\lambda_n)
$.
Since $\varepsilon$ is arbitrary, we conclude that
$\displaystyle |([A,B]v,v)| \le \liminf_{n \to \infty}b(\lambda_n)$.
\end{proof}
\begin{proof}[Proof of (ii) of {\rm{Theorem~\ref{theorem:refinementtsirelsoninequality}}}:]
If $F(A+B) = \emptyset$, it is straightforward that $\displaystyle \sup_{\lambda \in \tilde{F}(A+B, \mathcal{A})}b(\lambda) \le \rho([A,B])$.
Thus we may assume that $F(A+B) \ne \emptyset$.
Let $\lambda \in F(A+B)$.
Let $\{\lambda_i\}_{i=1}^\infty$ be a sequence in $(0,2)$, and let $\{u_i\}_{i=1}^\infty$ and $\{v_i\}_{i=1}^\infty$ be sequences
 of unit vectors in $V$ such that
\begin{itemize}
\item For any $i \in \mathbb{N}$, $-\lambda_i$ and $\lambda_i$ are eigenvalues of $A_i+B_i$.
\item For any $i \in \mathbb{N}$, $u_i$ and $v_i$ are  eigenvectors of $A_{i}+B_{i}$
corresponding to $-\lambda_i$ and $\lambda_i$, respectively.
\item The following convergence holds:
\begin{align*} 
\lim_{i\to \infty}\lambda_i=\lambda,
\lim_{i\to \infty}\max\{\|(A- A_i) u_i\|, \|(A- A_i) v_i\|\}=0.
\end{align*}
\end{itemize}

Let $i \in \mathbb{N}$.
Then all subdiagonal elements of $A_i+B_i$ are nonzero.
From Proposition~\ref{fact:tridiagonaleigenvaluesdistinctcondition},
the eigenvalues of $A_i+B_i$ are distinct.
From Proposition~\ref{fact:infinitedihedralgroupfinitedimensionalrepresentationdecomposition} 
and
Proposition~\ref{fact:infinitedihedralgroupfinitedimensionalrepresentation},
there exists a two dimensional irreducible invariant subspace $W$ under $\tau_{P_i,Q_i}$ such that
$(A_i+B_i)|W$ has eigenvalues $-\lambda_i, \lambda_i$
and $u_i, v_i \in W$.
From 
Proposition~\ref{fact:infinitedihedralgroupfinitedimensionalrepresentation},
there exist $\alpha_i, \beta_i \in \mathbb{C}$ such that
$|\alpha_i|^2+|\beta_i|^2=1$ and
\begin{equation*}
|([A_i,B_i]w_i, w_i)|=\rho([A_i|W,B_i|W])=b(\lambda_i),
\end{equation*}
where
$
w_i=\alpha_i u_i+\beta_i v_i
$.
From the definition of the spectral radius of $[A,B]$,
we have $|([A,B]w_i, w_i)| \le \rho([A, B])$.
Let $\varepsilon>0$.

We show that,
for any sufficiently large $i \in \mathbb{N}$,
the inequality
$
|[A,B]w_i-[A_i,B_i]w_i| < \varepsilon
$ holds.
From the definition of $\{u_i\}_{i=1}^\infty$ and $\{v_i\}_{i=1}^\infty$, there exists $i_0 \in \mathbb{N}$ such that
for any $i \in [i_0,\infty) \cap \mathbb{N}$
\begin{equation*}
\max\{\|(A-A_i)u_i\|, \|(B-B_i)u_i\|,\|(A-A_i)v_i\|, \|(B-B_i)v_i\|\} < \dfrac{\varepsilon}{8}.
\end{equation*}
Let $i \in [i_0,\infty) \cap \mathbb{N}$.
Since $|\alpha_i|,|\beta_i| \le 1$, it follows that
\begin{equation*}
\|B w_i-B_i w_i\| \le |\alpha_i|\|(B-B_i)u_i\|+|\beta_i|\|(B-B_i)v_i\| < \dfrac{\varepsilon}{4}. 
\end{equation*}
Because $A$ is unitary, we obtain
\begin{equation*}
\|AB w_i-AB_i w_i\| < \dfrac{\varepsilon}{4}. 
\end{equation*}
Since $W$ is invariant under $B_i$ and $B_i$ is unitary,
we have $B_i w_i \in W$ and $\|B_i w_i\| = 1$.
From 
Proposition~\ref{fact:infinitedihedralgroupfinitedimensionalrepresentation},
there exists $\alpha_i',\beta_i' \in \mathbb{C}$ such that
\begin{equation*}
B_i w_i=\alpha_i' u_i + \beta_i' v_i,\quad |\alpha_i'|^2+|\beta_i'|^2 =1.
\end{equation*}
Since $|\alpha_i'|,|\beta_i'| \le 1$, it follows that
\begin{equation*}
\|AB_i w_i-A_iB_i w_i\| \le |\alpha_i'|\|(A-A_i)u_i\|+|\beta_i'|\|(A-A_i)v_i\| < \dfrac{\varepsilon}{4}. 
\end{equation*}
Therefore, we obtain $\|ABw_i-A_iB_iw_i\| < \dfrac{\varepsilon}{2}$.
Following the same reasoning as above, we have 
$\|BAw_i-B_iA_iw_i\| < \dfrac{\varepsilon}{2}$.
We conclude that $\|[A,B]w_i-[A_i,B_i]w_i\| < \varepsilon$.

Thus we have
\begin{equation*}
b(\lambda_i)-\varepsilon=|([A_i,B_i]w_i,w_i)|-\varepsilon \le |([A,B]w_i,w_i)| \le \rho([A,B]).
\end{equation*}
Since $b$ is continuous, by letting $i \to \infty$, we have
$
b(\lambda)-\varepsilon \le \rho([A,B])
$.
Since $\varepsilon$ is arbitrary, we conclude that $b(\lambda) \le \rho([A,B])$.
\end{proof}
\begin{proof}[Proof of (iii) of {\rm{Theorem~\ref{theorem:refinementtsirelsoninequality}}}:]
It follows from Lemma~\ref{lemma:finitedimensionalspectrumsumtrivial}.
\end{proof}

\subsection{Equality conditions for the spectral radius estimates}
\label{subsec:equality}

The spectral radius of $[A(\theta),B(\theta)]$ for $\theta \in (0,\pi)$ is given below.

\begin{Example}
\label{example:constantangleoneshiftedform}
The following statements hold:
\begin{enumerate}
\item[(i)] The set $F(A(\theta)+B(\theta))$ satisfies
\begin{equation*}
(0,2\sin\theta) \subset F(A(\theta)+B(\theta)) \subset
(0,2\sin\theta].
\end{equation*}
\item[(ii)] The following equality holds:
\begin{equation*}
\displaystyle \sup_{\lambda \in F(A(\theta)+B(\theta))}b(\lambda)=\lim_{n \to \infty}b(\lambda_n)=\rho([A(\theta),B(\theta)]),
\end{equation*}
where $\lambda_{n}$ is an eigenvalue of $A_{n}(\theta)+B_{n}(\theta)$
whose square is closest to $2$ for $n \in \mathbb{N}$.
\item[(iii)] For any $\theta_1 \ne \theta_2 \in (0,\pi)$,
the unitary representation
$\tau_{P(\theta_1),Q(\theta_1)}$ is not unitarily equivalent
to the unitary representation $\tau_{P(\theta_2),Q(\theta_2)}$.
\item[(iv)] We have
\begin{equation*}
\rho([A(\theta),B(\theta)])=
\left\{
\begin{array}{ll}
2|\sin(2\theta)| & \quad |\theta - \dfrac{\pi}{2}| > \dfrac{\pi}{4}, \vspace{2mm}\\
2 & \quad |\theta - \dfrac{\pi}{2}| \le \dfrac{\pi}{4}.
\end{array}
\right.
\end{equation*}
\end{enumerate}
\end{Example}

The proof of this example is given in
\ifthenelse{\value{arxiv} = 1}{
 \ref{appndix:constantangleoneshiftedform}.
}{
 \cite{YFujii2025}.
}

\subsection{Application: Estimating spectral radius of Bell-CHSH operator}
\label{subsec:chsh}

For bipartite quantum systems, we use the following notation.

\begin{Notation}
\label{notation:bipartiatesystem}
Let
\begin{enumerate}
\item[(S1)] $V_1$ and $V_2$ be separable Hilbert spaces,
\item[(S2)] $P_{i,j}$ be orthogonal projections on $V_i$ for $i,j \in \{1,2\}$,
\item[(A1)] $(V_i,\{P_{i,1},P_{i,2}\})$ satisfies either {\rm{Case~1}} or {\rm{Case~2}} for any $i \in \{1,2\}$, and
\item[(S3)] for $i \in \{1,2\}$,
let $\mathcal{A}_i=\{(V_{i,n},\{P_{i,1,n},P_{i,2,n}\})\}_{n=1}^\infty$ be a finite dimensional approximation of $(V_i,\{P_{i,1},P_{i,2}\})$,
\item[(S4)] for $i \in \{1,2\}$,
when $(V_i,\{P_{i,1},P_{i,2}\})$ satisfies {\rm{Case~1}}, let $\mathcal{A}_i$ be the approximation defined in {\rm{Notation~\ref{notation:finitedimensionalspectrumsumtrivial}}},
\item[(S5)] for $i \in \{1,2\}$,
when $(V_i,\{P_{i,1},P_{i,2}\})$ satisfies {\rm{Case~2}}, let $\mathcal{A}_i$ be the approximation defined in {\rm{Theorem~\ref{theorem:twodimensionalapproximation}}}.
\end{enumerate}
Under these assumptions, we define $A_i, B_i,A_{i,n},B_{i,n},\lambda_{i,n}$ $(i \in \{1,2\}, n \in \mathbb{N})$
as follows:
\begin{enumerate}
\item[(i)] $A_i=2P_{1,i}-\operatorname{id}_{V_1}, B_i=2P_{2,i}-\operatorname{id}_{V_2}$,
\item[(ii)] $A_{i,n}=2P_{1,i,n}-\operatorname{id}_{V_{1,n}}$,
\item[(iii)] $B_{i,n}=2P_{2,i,n}-\operatorname{id}_{V_{2,n}}$,
\item[(iv)] $\lambda_{1,n}$ is an eigenvalue of $A_{1,n}+A_{2,n}$
whose square is closest to $2$,
\item[(v)] $\lambda_{2,n}$ is an eigenvalue of $B_{1,n}+B_{2,n}$ whose square is closest to $2$.
\end{enumerate}
\end{Notation}

From Theorem~\ref{theorem:refinementtsirelsoninequality},
we obtain upper and lower bounds for the spectral radius of the Bell-CHSH operator,
as follows.

\begin{Theorem}
\label{theorem:chsh}
Assume the setups (S1)-(S5), assumption (A1), notation (i)-(v) in {\rm{Notation~\ref{notation:bipartiatesystem}}}.
Under these assumptions, we have
\begin{equation*}
\sup_{\substack{\lambda \in \tilde{F}(A_1+A_2),\\\lambda' \in \tilde{F}(B_1+B_2)}}\sqrt{4+b(\lambda)b(\lambda')} \le \rho(\mathcal{B}) \le \liminf_{n \to \infty}\sqrt{4+b(\lambda_{1,n})b(\lambda_{2,n})}.
\end{equation*}
\end{Theorem}

\section*{Acknowledgements}
We thank Professors Albrecht B\"ottcher and Ilya M. Spitkovsky for their helpful comments and for sharing valuable insights relevant to this work.
They kindly pointed out minor errors in the earlier version of our preprint \cite{YFujii2025} and helped us correct them, for which we are deeply grateful.
Their observations greatly improved both the accuracy and the clarity of the current manuscript.

We also thank the anonymous referee for the careful reading of our manuscript and for the insightful comments that significantly improved the exposition and helped us better emphasize the contribution of our decomposition formula.

Y. F. would like to acknowledge Koichi Tojo, Takayuki Okuda, Yuichiro Tanaka, Yoji Koyama, and Kohei Hatano for their valuable comments and suggestions.

\ifthenelse{\value{arxiv} = 1}{
\appendix
\section{Proof of Example~\ref{example:constantangleoneshiftedform}}
\label{appndix:constantangleoneshiftedform}

As a preparation, we prove some properties of the eigenvalues of $A_n(\theta)+B_n(\theta)$ for $n \in \mathbb{N}$.

\begin{apxlemma}
\label{lemma:constantshiftedform}
Given $\theta \in (0,\pi)$, let
\begin{enumerate}
\item[(S1)] Functions $f_n:\mathbb{R}\setminus\left\{\dfrac{k}{n+1}\pi\middle|k \in \mathbb{Z}\right\} \to \mathbb{R}$ for $n \in \mathbb{N}$ be defined by
\begin{equation*}
f_n(\phi)=\dfrac{\sin((n-1)\phi)}{\sin((n+1)\phi)}\quad (\phi \in \mathbb{R}),
\end{equation*}
\item[(S2)] Intervals $I_{n,k}$ for $n \in \mathbb{N}$ and $k \in [0,n] \cap \mathbb{Z}$ be defined by
\begin{equation*}
I_{n,k}=\left(\dfrac{k}{n+1}\pi,\dfrac{k+1}{n+1}\pi\right), \text{ and}
\end{equation*}
\item[(S3)] Sets $Z_n$ for $n \in \mathbb{N}$ be defined by
\begin{equation*}
Z_n=\left\{\phi \in \mathbb{R}\middle| (1+\cos\theta)^2\sin((n-1)\phi)=\sin^2(\theta)\sin((n+1)\phi)\right\},
\end{equation*}
\end{enumerate}
then the following statements hold:
\begin{enumerate}
\item[(i)] For any $n \in \mathbb{N}$ and $\phi \in [0,\pi]$,
\begin{equation*}
f_{2n}(\phi)=f_{2n}(\pi-\phi)=f_{2n}(\pi+\phi).
\end{equation*}
\item[(ii)] For any sufficiently large $n \in \mathbb{N}$, $f_{2n}$ is monotone in each of the following two intervals:
\begin{equation*}
\left(0,\dfrac{\pi}{2(2n+1)}\right),\left(\pi-\dfrac{\pi}{2(2n+1)},\pi\right).
\end{equation*}
\item[(iii)] For any sufficiently large $n \in \mathbb{N}$ and $k \in ([1,2n-1] \cap \mathbb{N})\setminus \{n\}$, $f_{2n}$ is monotone in each of $\left[\dfrac{\pi}{2(2n+1)},\dfrac{\pi}{2n+1}\right)$, $I_{2n,k}$,\\ $\left(\dfrac{2n\pi}{2n+1},\pi-\dfrac{\pi}{2(2n+1)}\right]$.
\item[(iv)] Assume that $\theta \ne \dfrac{1}{2}\pi$.
Then, for any sufficiently large $n \in \mathbb{N}$,
\begin{align*}
\#(Z_{2n} \cap I_{2n,0}) \le 3, \#(Z_{2n} \cap I_{2n,2n}) \le 3, Z_{2n} \cap I_{2n,n}=\emptyset,
\end{align*}
and $\{0,\pi\} \subset Z_{2n}$,
and for any $k \in ([1,2n-1] \cap \mathbb{N})\setminus \{n\}$
\begin{align*}
\#(Z_{2n} \cap I_{2n,k})=1.
\end{align*}
\item[(v)] Assume that $\theta = \dfrac{1}{2}\pi$.
Then
\begin{equation*}
Z_{2n} \cap [0,\pi] \subset \left\{0,\pi\right\} \cup \left\{\dfrac{1}{4n}\pi,\dfrac{3}{4n}\pi,\dots,\dfrac{4n-1}{4n}\pi\right\}.
\end{equation*}
\item[(vi)] Assume that $\theta \ne \dfrac{1}{2}\pi$. Then 
for any sufficiently large $n \in \mathbb{N}$,
\begin{align*}
&\#(\sigma(A_{n}(\theta)+B_{n}(\theta)) \cap 2\cos(I_{2n,0})\sin\theta) \le 1,\\
&\#(\sigma(A_{n}(\theta)+B_{n}(\theta)) \cap 2\cos(I_{2n,2n})\sin\theta) \le 1,\\
&\sigma(A_{n}(\theta)+B_{n}(\theta)) \cap 2\cos(I_{2n,n})\sin\theta=\emptyset,
\end{align*}
and for any $k \in ([1,2n-1] \cap \mathbb{N}) \setminus \{n\}$,
\begin{align*}
\#(\sigma(A_{n}(\theta)+B_{n}(\theta)) \cap 2\cos(I_{2n,k})\sin\theta) = 1.
\end{align*}
\item[(vii)] Assume that $\theta = \dfrac{1}{2}\pi$. Then for any $n \in \mathbb{N}$
\begin{align*}
&\sigma(A_{n}(\theta)+B_{n}(\theta)) \subset \{-2,2\}\cup\left\{2\cos\left(\dfrac{2k-1}{4n}\pi\right)\middle|k=1,2,\dots,2n\right\}.
\end{align*}
\item[(viii)] Let $\varepsilon>0$ and $\lambda_0 \in (-2\sin\theta,2\sin\theta)$ and $\eta > 0$
such that $|\lambda_0| + \eta<2\sin\theta$.
Then there exist $n_0 \in \mathbb{N}$ such that
for any $n \in [n_0,\infty) \cap \mathbb{N}$
and any eigenvalue $\lambda$ of $A_{n}(\theta)+B_{n}(\theta)$ such that $\lambda \in (\lambda_0-\eta,\lambda_0+\eta)$, and for any unit eigenvector $u$ corresponding to $\lambda$, we have $|u_{2n}|<\varepsilon$.
\item[(ix)] Assume that $\theta < \dfrac{1}{2}\pi$.
Let $\varepsilon>0$. Then for any sufficiently large $n \in \mathbb{N}$, there exists $\lambda \in [2-\varepsilon,2]$ such that
\begin{equation*}
\sigma(A_n(\theta)+B_n(\theta)) \cap [-2\sin\theta,2\sin\theta]^c=\{-\lambda,\lambda\}.
\end{equation*}
\item[(x)] Assume that $\theta > \dfrac{1}{2}\pi$.
Then for any sufficiently large $n \in \mathbb{N}$,
\begin{align*}
&\#(\sigma(A_{n}(\theta)+B_{n}(\theta)) \cap 2\cos(I_{2n,0})\sin\theta) = 1,\\
&\#(\sigma(A_{n}(\theta)+B_{n}(\theta)) \cap 2\cos(I_{2n,2n})\sin\theta) = 1,\\
&\sigma(A_n(\theta)+B_n(\theta)) \subset [-2\sin\theta,2\sin\theta].
\end{align*}
\end{enumerate}
\end{apxlemma}
\begin{proof}[Proof of (i)]
From the fundamental trigonometric identity, we have
\begin{align*}
\sin((2n+1)(\pi-\phi))&=\sin((2n+1)\phi),\\
\sin((2n-1)(\pi-\phi))&=\sin((2n-1)\phi),\\
\sin((2n+1)(\pi+\phi))&=(-1)\sin((2n+1)\phi),\\
\sin((2n-1)(\pi+\phi))&=(-1)\sin((2n-1)\phi).
\end{align*}
Thus, (i) holds.
\end{proof}
\begin{proof}[Proof of (ii)]
Let $n \in \mathbb{N}$.
From statement (i), it is enough to prove that
$f_{2n}$ is monotone in $\left(0,\dfrac{\pi}{2(2n+1)}\right)$.
Let $\phi \in \left(0,\dfrac{\pi}{2(2n+1)}\right)$.
Then we have
\begin{align*}
f_{2n}^{\prime}(\phi)
&=\dfrac{1}{\sin^2((2n+1)\phi)}\{(2n-1)\cos((2n-1)\phi)\sin((2n+1)\phi)\\
&\quad-(2n+1)\sin((2n-1)\phi)\cos((2n+1)\phi)\}\\
&=
\dfrac{(2n-1)\cos((2n-1)\phi) \sin((2n+1)\phi)}{\sin^2((2n+1)\phi)}
\left\{
1-\dfrac{b\tan(a\phi)}{a\tan(b\phi)}
\right\},
\end{align*}
where $a=2n-1, b=2n+1$.
There exists a sequence $\{c_i\}_{i=0}^\infty$ of positive numbers
such that
for any constant $C>0$ and $x \in \left[0,\dfrac{\pi}{2C}\right)$,
$\displaystyle
\tan(Cx)=\sum_{i=0}^\infty c_i C^{2i+1}x^{2i+1}
$.
Therefore, we have $1-\dfrac{b\tan(a\phi)}{a\tan(b\phi)} > 0$.
Thus $f_{2n}$ is monotone in $\left(0,\dfrac{\pi}{2(2n+1)}\right)$.
\end{proof}
\begin{proof}[Proof of (iii)]
Let $\phi$ be a number in
\begin{equation*}
\left[\dfrac{\pi}{2(2n+1)},\dfrac{\pi}{2n+1}\right) \cup \bigcup_{k \in \{1,\dots,2n-1\}\setminus\{n\}}I_{2n,k} \cup \left(\dfrac{2n\pi}{2n+1},\pi-\dfrac{\pi}{2(2n+1)}\right].
\end{equation*}
Then, the following equality holds:
\begin{align*}
&f_{2n}^{\prime}(\phi)\\
&=
\dfrac{1}{\sin^2((2n+1)\phi)}
\{
(2n-1)\cos((2n-1)\phi) \sin((2n+1)\phi)
\\
&\quad
-
(2n+1)\sin((2n-1)\phi)\cos((2n+1)\phi)
\}\\
&=
\dfrac{(2n-1)}{\sin^2((2n+1)\phi)}
\{
\cos((2n-1)\phi) \sin((2n+1)\phi)\\
&\quad-
\sin((2n-1)\phi)\cos((2n+1)\phi)\\
&\quad
+\left(1-\dfrac{(2n+1)}{(2n-1)}\right)\sin((2n-1)\phi)\cos((2n+1)\phi)
\}\\
&=
\dfrac{1}{\sin^2((2n+1)\phi)}
\left\{
(2n-1)\sin(2\phi)
-
2\sin((2n-1)\phi)\cos((2n+1)\phi)
\right\}.
\end{align*}
Here,
\begin{equation*}
\dfrac{1}{2n+1}\pi \le 2\phi \le \dfrac{2n}{2n+1}\pi\quad
\text{or}\quad
\dfrac{2n+2}{2n+1}\pi \le 2\phi \le 2\pi-\dfrac{\pi}{2n+1}.
\end{equation*}
Hence, we have
\begin{align*}
|(2n-1)\sin(2\phi)| \ge \left(2n-1\right)
\min\left\{a,b,c,d\right\},
\end{align*}
where $a,b,c$ are defined by
\begin{align*}
a&=\left|\sin\left(\dfrac{1}{2n+1}\pi\right)\right|,
b=\left|\sin\left(\dfrac{2n}{2n+1}\pi\right)\right|,\\
c&=\left|\sin\left(\dfrac{2n+2}{2n+1}\pi\right)\right|,
d=\left|\sin\left(2\pi-\dfrac{\pi}{2n+1}\right)\right|
\end{align*}
From the proof of (i), we have
\begin{align*}
&\left|\sin\left(\dfrac{2n}{2n+1}\pi\right)\right|=
\left|\sin\left(\pi-\dfrac{1}{2n+1}\pi\right)\right|
=\left|\sin\left(\dfrac{1}{2n+1}\pi\right)\right|,\\
&\left|\sin\left(\dfrac{2n+2}{2n+1}\pi\right)\right|=
\left|\sin\left(\pi+\dfrac{1}{2n+1}\pi\right)\right|
=\left|\sin\left(\dfrac{1}{2n+1}\pi\right)\right|.
\end{align*}
Since for any sufficiently large $n \in \mathbb{N}$ 
\begin{align*}
\left(2n-1\right)\sin\left(\dfrac{1}{2n+1}\pi\right)
=
\sin\left(\dfrac{1}{2n+1}\pi\right)
\dfrac{2n+1}{\pi}
\dfrac{2n-1}{2n+1}
\pi
>2,
\end{align*}
we have $|(2n-1)\sin(2\phi)| > 2$.
Since $|2\sin((2n-1)\phi)\cos((2n+1)\phi)| \le 2$, 
for any sufficiently large $n \in \mathbb{N}$,
$f_{2n}$ is monotone in each of the following intervals:
\begin{align*}
&\left[\dfrac{\pi}{2(2n+1)},\dfrac{\pi}{2n+1}\right),I_{2n,1},\dots,I_{2n,n-1},I_{2n,n+1},\dots,I_{2n,2n-1},\\
&\left(\dfrac{2n\pi}{2n+1},\pi-\dfrac{\pi}{2(2n+1)}\right].
\end{align*}
\end{proof}
\begin{proof}[Proof of (iv)]
Since
$\dfrac{\sin^2\theta}{(1+\cos\theta)^2}
=
\tan^2\left(\dfrac{\theta}{2}\right)
$,
we have $\dfrac{\sin^2 \theta}{(1+\cos\theta)^2} \ne 1$.
For any sufficiently large $n \in \mathbb{N}$,
\begin{align*}
\dfrac{2n-1}{2n+1} \ne \dfrac{\sin^2 \theta}{(1+\cos\theta)^2},\quad
\sin\left(\dfrac{\pi}{2n+1}\right)\dfrac{2n+1}{\pi}\dfrac{2n-1}{2n+1}\pi>2,
\end{align*}
and for any $\phi \in I_{2n,n}$
$
|f_{2n}(\phi)-1|<1-\dfrac{\sin^2 \theta}{(1+\cos\theta)^2}
$.

Fix such an $n \in \mathbb{N}$.
It follows that $Z_{2n} \cap I_{2n,n} = \emptyset$.

From statements (ii) and (iii), we have
$\#(Z_{2n} \cap I_{2n,0}), \#(Z_{2n} \cap I_{2n,2n}) \le 3$.

Let $k \in \{1,\dots,2n-1\}\setminus\{n\}$.
From statement (iii) and
\begin{align*}
\lim_{\eta \to \frac{k}{2n+1}+0}|f_{2n}(\eta)|=
\lim_{\eta \to \frac{k+1}{2n+1}-0}|f_{2n}(\eta)|=
\infty,\\
\lim_{\substack{\eta_1 \to \frac{k}{2n+1}+0, \\ \eta_2 \to \frac{k+1}{2n+1}-0}}f_{2n}(\eta_1)f_{2n}(\eta_2)=-\infty,
\end{align*}
we have $\#(Z_{2n} \cap I_{2n,k})=1$.
\end{proof}
\begin{proof}[Proof of (v)]
Let $\phi \in Z_{2n}$.
Since
$\dfrac{\sin^2 \theta}{(1+\cos \theta)^2}=1$, we have
\begin{equation*}
\sin((2n-1)\phi)=\sin((2n+1)\phi).
\end{equation*}
Therefore, there exists $k \in \mathbb{Z}$ such that
\begin{equation*}
4n \phi = (2k+1) \pi\text{ or }
2 \phi = 2k \pi.
\end{equation*}
Hence, statement (v) holds.
\end{proof}
\begin{proof}[Proof of (vi)-(viii)]
We will first prove (vi) and (vii).
In order to do so,
we apply the result of Yueh and Cheng \cite{WCYueh2008}, which is introduced in Section~\ref{subsubsec:toeplitz}, to the symmetric tridiagonal matrix
\begin{equation*}
A(\theta)+B(\theta)
=
\begin{pmatrix}
    1+\cos\theta & \sin\theta & 0 & \dots & 0 & 0\\
        \sin\theta & 0 & \sin\theta & \dots & 0 & 0\\
                     0 & \sin\theta &  0 & \dots & 0 & 0\\
                    \dots & \dots & \dots & \dots & \dots & \dots\\
                     0 & 0 & 0 & \dots & 0 & \sin\theta\\
                     0 & 0 & 0 & \dots & \sin\theta & -(1+\cos\theta)
\end{pmatrix}.
\end{equation*}
Let $a,b,c,\alpha,\beta,\gamma,\delta, \mu$ be complex numbers defined in Section~\ref{subsubsec:toeplitz}.
Let $n \in \mathbb{N}$.
In the present setting,
\begin{align*}
&a=c=\sin\theta \ne 0,\quad\alpha=\beta=b=0,\quad\gamma=-\delta=1+\cos(\theta),\quad\mu=1.
\end{align*}
Then, if we find $\phi \in [0,2\pi)$ such that $\sin\phi\ne0$ and
$\phi$ satisfies
\begin{eqnarray*}
\sin^2\theta\sin((2n+1)\phi)-(1+\cos\theta)^2\sin(2n-1)\phi=0,
\end{eqnarray*}
$\lambda=2\sin\theta \cos\phi \in \sigma(A_n(\theta)+B_n(\theta))$.

From Proposition~\ref{fact:tridiagonaleigenvaluesdistinctcondition},
the eigenvalues of $A_n(\theta)+B_n(\theta)$ are distinct for any $n \in \mathbb{N}$.
Moreover, from statements (iv) and (v), statements (vi) and (vii) hold.

Next, we will prove (viii).
Let $\lambda_0 \in (-2\sin\theta,2\sin\theta)$ and $\eta > 0$
such that $|\lambda_0| + \eta<2\sin\theta$.
Define $r$, $J$, $R:[0,2\pi] \to \mathbb{R}$,
$R_0$, $C$ as follows:
\begin{align*}
r&=\dfrac{\sin\theta}{1+\cos\theta}, \quad R(\phi)=2((r-\cos \phi)^2+\sin^2\phi)\quad(\phi \in [0,2\pi]),\\
J&=\left\{\phi \in [0,2\pi]\middle|2\sin\theta\cos\phi \in [\lambda_0-\eta,\lambda_0+\eta]\right\},R_0=\inf \left(R\left(J\right)\right),\\
C&=\inf\left\{|1-\exp(i2\phi)|\ \middle|\phi \in J\right\}.
\end{align*}
Clearly $\sin\phi\ne 0$ for any $\phi \in J$, and $R_0$ and $C$ are positive.

Let $\lambda \in (\lambda_0-\eta,\lambda_0+\eta)$.
In the present setting,
for any eigenvalue $\lambda$ of $A_n(\theta)+B_n(\theta)$
and any eigenvector $u$ of $X$ corresponding to $\lambda$,
there exists $\phi \in J$ such that
\begin{align}
\label{eq:yueh33implication}
u_j&=\dfrac{2u_1i}{\sqrt{\omega}}\left\{\sin\theta\sin(j\phi)-(1+\cos\theta)\sin((j-1)\phi)\right\} \notag\\
&=-\dfrac{2u_1i}{\sqrt{\omega}}(1+\cos\theta)(\sin((j-1)\phi)-r\sin(j\phi))\quad(j=1,2,\dots,2n).
\end{align}
Then we have
\begin{align*}
&\quad-4\sum_{k=0}^{2n-1}(\sin (k\phi) - r \sin ((k+1)\phi))^2\\
&=\sum_{k=0}^{2n-1}(\exp (ik\phi)-\exp(-ik\phi)-r(\exp (i(k+1)\phi)-\exp(-i(k+1)\phi)))^2\\
&=\sum_{k=0}^{2n-1}(\exp (ik\phi)-\exp(-ik\phi)-r\exp (i(k+1)\phi)+r\exp(-i(k+1)\phi))^2\\
&=\sum_{k=0}^{2n-1}(\exp (i 2k\phi) + \exp (i(-2k\phi)) + r^2 \exp (i2(k+1)\phi)+r^2\exp(-i2(k+1)\phi)\\
&\quad-2-2r \exp (i (2k+1)\phi)+2r\exp (-i\phi)+2r\exp (i\phi)
-2r\exp (-i (2k+1)\phi)\\
&\quad-2r^2)\\
&=\sum_{k=0}^{2n-1}(\exp (i(2k\phi)) + \exp (i(-2k\phi)) + r^2 \exp (i2(k+1)\phi)+r^2\exp(-i2(k+1)\phi)\\
&\quad-2r \exp (i (2k+1)\phi)-2r\exp (-i (2k+1)\phi)-2(r^2-2(\cos\phi) r+1))\\
&=-2nR(\phi)+\dfrac{\exp (i4n\phi)-1}{\exp (i2\phi)-1}+\dfrac{\exp(-i4n\phi)-1}{\exp(-i2\phi)-1}+r^2\exp (i\phi)\dfrac{\exp (i4n\phi)-1}{\exp (i2\phi)-1}\\
&\quad+r^2\exp(-i\phi)\dfrac{\exp (-i4n\phi)-1}{\exp(-i2\phi)-1}-2r\exp (i\phi)\dfrac{\exp (i4n\phi)-1}{\exp (i2\phi)-1}\\
&\quad-2r\exp(-i\phi)\dfrac{\exp (-i4n\phi)-1}{\exp(-i2\phi)-1}\\
&=-2nR(\phi)+2\operatorname{Re}\left((\exp (i \phi) r^2-2\exp (i \phi)r+1)\dfrac{\exp (i4n\phi)-1}{\exp (i2\phi)-1}\right).
\end{align*}

Hence, for any sufficiently large $n \in \mathbb{N}$
\begin{align*}
&\left\|-2nR(\phi)+2\operatorname{Re}\left((\exp (i \phi) r^2-2\exp (i \phi)r+1)\dfrac{\exp (i4n\phi)-1}{\exp (i2\phi)-1}\right)\right\|\\
&\ge 2nR_0 - \dfrac{4(r+1)^2}{C}.
\end{align*}
It follows that $\dfrac{|u_{2n}|}{\|u\|} \to 0\quad\left(n \to \infty\right)$.
Consequently, statement (viii) holds.
\end{proof}
\begin{proof}[Proof of (ix)]
In Theorem~4 in \cite{ABottcher2026},
it is shown that
$2 \in \sigma(A(\theta)+B(\theta))$.
Hence, 
there exists a unit vector $u \in \ell^2$ such that
$|((A(\theta)+B(\theta))u,u)|>2-\dfrac{\varepsilon}{2}$.
From Proposition~\ref{prop:oneshiftedcyclicrepresentation},
for any sufficiently large $n \in \mathbb{N}$ there exists a unit vector $u_n \in \mathbb{C}^{2n}$ such that
$|((A_n(\theta)+B_n(\theta))u_n,u_n)|>2-\varepsilon$.
Fix such $n \in \mathbb{N}$.
Then $(A_n(\theta)+B_n(\theta))$ has an eigenvalue $\lambda \in [-2,-2+\varepsilon]\cup[2-\varepsilon,2]$.

From (vi),
$\#(\sigma(A_n(\theta)+B_n(\theta)) \cap [-2\sin\theta,2\sin\theta]) \ge 2n-2$.
Moreover, from Proposition~\ref{fact:infinitedihedralgroupfinitedimensionalrepresentationdecomposition}
and
Proposition~\ref{fact:infinitedihedralgroupfinitedimensionalrepresentation},
 $-\gamma \in (\sigma(A_n(\theta)+B_n(\theta)) \cap [-2\sin\theta,2\sin\theta])$ for any $\gamma \in(\sigma(A_n(\theta)+B_n(\theta)) \cap [-2\sin\theta,2\sin\theta])$.
Since $\operatorname{tr}(A_n(\theta)+B_n(\theta))=0$,
$-\lambda \in A_n(\theta)+B_n(\theta)$.
Consequently, we have
\begin{equation*}
\sigma(A_n(\theta)+B_n(\theta)) \cap [-2\sin\theta,2\sin\theta]^c=\{-\lambda,\lambda\}.
\end{equation*}
\end{proof}
\begin{proof}[Proof of (x)]
Note that
\begin{align*}
\lim_{n \to \infty}f_{2n}\left(\dfrac{\pi}{2(2n+1)}\right)=\lim_{n \to \infty}f_{2n}\left(\pi-\dfrac{\pi}{2(2n+1)}\right)=1,
\end{align*}
and for any $n \in \mathbb{N}$
\begin{equation*}
\lim_{\phi \to \frac{\pi}{2n+1}-0}f_{2n}(\phi)=\lim_{\phi \to \left(\pi-\frac{\pi}{2n+1}\right)+0}f_{2n}(\phi)=\infty.
\end{equation*}
Since for any $\theta \in \left(\dfrac{\pi}{2},\pi\right)$
\begin{equation*}
\dfrac{\sin^2\theta}{(1+\cos\theta)^2}=\tan^2\left(\dfrac{\theta}{2}\right)>1,
\end{equation*}
from (iii), we have that  for any sufficiently large $n \in \mathbb{N}$
\begin{equation*}
1 \le \min\{\#(Z_{2n} \cap I_{2n,0}),\#(Z_{2n}\cap I_{2n,2n})\}.
\end{equation*} 
Fix such $n \in \mathbb{N}$.
Combining this and (vi), it follows that
\begin{equation*}
\sigma(A_n(\theta)+B_n(\theta)) \subset [-2\sin\theta,2\sin\theta].
\end{equation*}
\end{proof}

We give the proof of Example~\ref{example:constantangleoneshiftedform} as follows.

\begin{proof}[Proof of (i) of {\rm{Example~\ref{example:constantangleoneshiftedform}}}]
Let $\lambda \in (0,2\sin\theta)$.
From statements (vi) and (vii) of Lemma~\ref{lemma:constantshiftedform},
there exists $\{\eta_n\}_{n=1}^\infty \subset (-2\sin\theta,2\sin\theta)$ and 
$\{u_n\}_{n=1}^\infty$ such that $\displaystyle \lim_{n \to \infty}\eta_n = \lambda$, and for any $n \in \mathbb{N}$, $u_n$ is a unit eigenvector of $A_n(\theta)+B_n(\theta)$ corresponding to $\eta_n$.
Then for any $n \in \mathbb{N}$,
$
P_n(\theta)u_n=P(\theta)u_n, \|(Q_n(\theta)-Q(\theta))u_n\| \le 2|u_{n,2n}|
$,
where $u_{n,2n}$ is the $2n$-th component of $u_n$.
From statement (viii) of Lemma~\ref{lemma:constantshiftedform}, we have
$
\displaystyle \lim_{n \to \infty} \|(Q_n(\theta)-Q(\theta))u_n\|=0
$.
It follows that $\lambda \in F(A(\theta)+B(\theta))$.
Hence we have $(-2\sin\theta,2\sin\theta) \subset F(A(\theta)+B(\theta))$.

We prove that $F(A(\theta)+B(\theta)) \subset [-2\sin\theta,2\sin\theta]$.
In Theorem~4 in \cite{ABottcher2026},
it is shown that
$\sigma(A(\theta)+B(\theta)) \subset [-2\sin\theta,2\sin\theta] \cup \{2\}$.
Since $F(A(\theta)+B(\theta)) \subset (0,2)$,
we have $F(A(\theta)+B(\theta)) \subset  (0,2\sin\theta]$.
\end{proof}
\begin{proof}[Proof of (ii) of {\rm{Example~\ref{example:constantangleoneshiftedform}}}]
From statements (vi)-(x) of \\Lemma~\ref{lemma:constantshiftedform}, we have
\begin{equation*}
\displaystyle \limsup_{n \to \infty}b(\lambda_n) \le \sup_{\lambda \in F(A(\theta)+B(\theta))}b(\lambda)
.
\end{equation*}
Moreover, from statement (ii) of Theorem~\ref{theorem:refinementtsirelsoninequality}, we have
\begin{equation*}
\displaystyle \sup_{\lambda \in F(A(\theta)+B(\theta))}b(\lambda)=\lim_{n \to \infty}b(\lambda_n)=\rho([A(\theta),B(\theta)]).
\end{equation*}
\end{proof}
\begin{proof}[Proof of (iii) of {\rm{Example~\ref{example:constantangleoneshiftedform}}}]
This follows directly from statement (i) and the fact that
$F(\tau(g_1)+\tau(g_2))$ is invariant under any unitary intertwining operator for any unitary representation $(\tau,V)$ of $D_\infty$.
\end{proof}
\begin{proof}[Proof of (iv) of Example~\ref{example:constantangleoneshiftedform}]
From statement (i), we have
\begin{equation*}
\sup_{\lambda \in F(A(\theta)+B(\theta))}b(\lambda)
=
\left\{
\begin{array}{ll}
2|\sin(2\theta)| & \quad|\theta - \dfrac{\pi}{2}| > \dfrac{\pi}{4}, \vspace{2mm}\\
2 & \quad |\theta - \dfrac{\pi}{2}| \le \dfrac{\pi}{4}.
\end{array}
\right.
\end{equation*}
Moreover, from statement (ii), statement (iv) holds.
\end{proof}
}

\end{document}